\documentclass[USenglish, a4paper, 10pt, oneside]{article}

	\usepackage{comment}
	\usepackage{lscape}
	\usepackage{amsmath}
	\usepackage[sort,nocompress]{cite}%increasing order for citations
    \usepackage[left=2cm, top=2.8cm, right=2cm, bottom=2.8cm, bindingoffset=0.5cm]{geometry}
	\usepackage{amsfonts}
	\usepackage{amssymb}
	\usepackage{calc}
	\usepackage{array}
	\usepackage{mathtools} % for ceiling function
	\usepackage{enumitem} %global enumerate format
	\setenumerate[1]{label= (\roman*)} %global enumerate format
	\usepackage{xparse} %optional package, used for commands in Preamble 
	\usepackage{etoolbox} % for \AtBeginEnvironment
	\usepackage{amsthm}      % For theorem styles and numbering
	
	\usepackage{hyperref}    % Enable hyperlinks in the document
	\hypersetup{
		colorlinks=true,
		linkcolor=blue,
		citecolor=blue,
		urlcolor=blue
	}
	\usepackage{todonotes}
	\usepackage[capitalize]{cleveref}
	\usepackage{thmtools}
	\usepackage{comment}
	\usepackage{graphicx}
	\usepackage{float}
	\usepackage{caption}
	\usepackage{pdflscape}
	\usepackage{rotating}
	\usepackage{algorithm,algpseudocode}

	\usepackage{pgfplots}
	\usepackage{tikz}
	\pgfplotsset{compat=newest}

	\theoremstyle{plain}
	\newtheorem{theorem}{Theorem}[section]
	
	\newtheorem{corollary}[theorem]{Corollary}
	
	\newtheorem{lemma}[theorem]{Lemma}
	
	\newtheorem{proposition}[theorem]{Proposition}

	\theoremstyle{definition}
	\newtheorem{definition}[theorem]{Definition}
	\newtheorem{example}[theorem]{Example}
	\newtheorem{remark}[theorem]{Remark}
	
	\crefname{theorem}{Theorem}{Theorems}
	\Crefname{theorem}{Theorem}{Theorems}
	\crefname{lemma}{Lemma}{Lemmas}
	\Crefname{lemma}{Lemma}{Lemmas}
	\crefname{fact}{Fact}{Facts}
	\Crefname{fact}{Fact}{Facts}
	\crefname{corollary}{Corollary}{Corollaries}
	\Crefname{corollary}{Corollary}{Corollaries}
	\crefname{proposition}{Proposition}{Propositions}
	\Crefname{proposition}{Proposition}{Propositions}
	\crefname{claim}{Claim}{Claims}
	\Crefname{claim}{Claim}{Claims}
	\crefname{assumption}{Assumption}{Assumptions}
	\Crefname{assumption}{Assumption}{Assumptions}
	\crefname{definition}{Definition}{Definitions}
	\Crefname{definition}{Definition}{Definitions}
	\crefname{example}{Example}{Examples}
	\Crefname{example}{Example}{Examples}
	\crefname{remark}{Remark}{Remarks}
	\Crefname{remark}{Remark}{Remarks}

	\newlist{enumerateq}{enumerate}{1}
	\setlist[enumerateq]{label=\textup{(\roman*)},
					ref=\thetheorem\textup{(\roman*)}}

	\AtBeginEnvironment{theorem}{\crefalias{enumerateqi}{theorem}}
	\AtBeginEnvironment{lemma}{\crefalias{enumerateqi}{lemma}}
	\AtBeginEnvironment{fact}{\crefalias{enumerateqi}{fact}}
	\AtBeginEnvironment{corollary}{\crefalias{enumerateqi}{corollary}}
	\AtBeginEnvironment{proposition}{\crefalias{enumerateqi}{proposition}}
	\AtBeginEnvironment{claim}{\crefalias{enumerateqi}{claim}}
	\AtBeginEnvironment{assumption}{\crefalias{enumerateqi}{assumption}}
	\AtBeginEnvironment{definition}{\crefalias{enumerateqi}{definition}}
	\AtBeginEnvironment{example}{\crefalias{enumerateqi}{example}}
	\AtBeginEnvironment{remark}{\crefalias{enumerateqi}{remark}}
	\newcommand{\lsubdiff}{\mathop{}\!\partial}			% limiting subdifferential
	\newcommand{\rsubdiff}{\mathop{}\!\widehat\partial}	% Frechet subdifferential
	\NewDocumentCommand{\psubdiff}{O{\lambda}}{\mathop{}\!\partial_p^{#1}} % Level prox subdifferential, need xsparse for optional argument. Use [] to add argument
	\newcommand{\func}[3]{#1:#2\rightarrow#3}
	\newcommand{\ffunc}[3]{#1:#2\rightrightarrows#3}
	\newcommand{\ip}[2]{\langle#1,#2\rangle} % inner product
	\newcommand{\norm}[1]{\left\lVert#1\right\rVert} %norm
	
	\NewDocumentCommand{\prox}{O{\lambda}}{\mathop{}\!\ensuremath{P}_{#1}} %Euclidean prox , need xsparse for optional argument. Use [] to add argument
	\NewDocumentCommand{\env}{O{\lambda}}{\mathop{}\!\ensuremath{e}_{#1}} %Euclidean envlope , need xsparse for optional argument. Use [] to add argument

	\NewDocumentCommand{\hull}{O{\lambda}}{\mathop{}\!\ensuremath{h}_{#1}} %Euclidean envlope , need xsparse for optional argument. Use [] to add argument

	\newcommand{\Id}{\ensuremath{\operatorname{Id}}}
	\newcommand{\sgn}{\ensuremath{\operatorname{sgn}}}
	\newcommand{\proj}{\ensuremath{\operatorname{Proj}}} %projector
	\newcommand{\defeq}{\coloneqq}
	\newcommand{\dist}{\ensuremath{\operatorname{dist}}} %distance function

	\DeclareMathOperator*{\argmin}{arg\,min}
	
	\NewDocumentCommand{\conv}{s}{%										%both for closed cvx hull. Put together intentionally to avoid conflict
	\IfBooleanTF{#1}%
	{\operatorname{\overline{conv}}}%
	{\operatorname{conv}}%
	}

	\NewDocumentCommand{\set}{s m o}{%
	\IfBooleanF{#1}{\left}\{
	#2
	\IfValueT{#3}{\left|#3\right.}
	\IfBooleanF{#1}{\right}\}
	}
	\NewDocumentCommand{\seq}{s m e{_^}}{%
	\IfBooleanF{#1}{\left}(
	#2
	\IfBooleanF{#1}{\right})
	\IfValueTF{#3}{_{#3}}{_{k\in\N}}
	\IfValueT{#4}{^{#4}}
	}

	\newcommand{\N}{\mathbb{N}}   % natrual numbers
	
	\newcommand{\Rn}{\ensuremath{\mathbb R^n}}
	\newcommand{\R}{\ensuremath{\mathbb R}}
	\newcommand{\RE}{\overline{\mathbb{R}}}
	\newcommand{\ball}{\mathbb{B}} % unit open ball
	
	\newcommand{\supp}{\ensuremath{\operatorname{supp}}}
	\newcommand{\spn}{\ensuremath{\operatorname{span}}}
	\newcommand{\card}{\ensuremath{\operatorname{card}}}
	
	\newcommand{\dom}{\ensuremath{\operatorname{dom}}}
	\newcommand{\ran}{\ensuremath{\operatorname{ran}}}
	\newcommand{\gph}{\ensuremath{\operatorname{gph}}}
	\newcommand{\gra}{\ensuremath{\operatorname{gph}}}
	
	\newcommand{\epi}{\ensuremath{\operatorname{epi}}}
\newenvironment{proofnoqed}[1][Proof]{% Define a custom proof environment without QED
	\noindent \textit{#1.} % Bold text for the proof title
}{%
}

	\newenvironment{proofitemize}[1][Proof]{%
	\begin{proofnoqed}[#1]
		\begin{itemize}[left=0pt, labelsep=0em, itemindent=0pt,label={}]
		}{%
		\end{itemize}
	\end{proofnoqed}
	}

	\newenvironment{newitemize}[1][Proof]
	{%
		\begin{itemize}[left=0pt, labelsep=0em, itemindent=0pt,label={}]
		}{%
		\end{itemize}
	}

    \usepackage{doi}
	\newcommand{\infconv}{\mathbin{\square}}
	\NewDocumentCommand{\T}{O{$\lambda$}}{\mathop{}\!\ensuremath{T}_{#1}} %prox-grad operator , need xsparse for optional argument. Use [] to add argument

	\usepackage{float}
	\usepackage{subcaption}
\begin{document}
\title{Variational convexity: new characterizations, calculus rules, and applications}

\author{
Radu Ioan Bo\c t\thanks{Faculty of Mathematics, University of Vienna, Oskar-Morgenstern-Platz 1, 1090 Vienna, Austria.
Emails:
\texttt{radu.bot@univie.ac.at},
\texttt{ziyuanw96@univie.ac.at}
}
\and
Ziyuan Wang\footnotemark[1]
}
\maketitle

\begin{abstract}
Introduced by R.T. Rockafellar in 2019, variational convexity is a generalized notion of convexity under which stationary points of nonconvex optimization problems can still be guaranteed to exhibit local optimality. In this paper, we provide characterizations of variationally convex functions through their proximal hulls and epigraphs, and investigate operations that preserve variational convexity, including nonlinear and linear composition, summation, and proximal averaging. We further apply these results to identify variational convexity in nonlinear programming problems with possibly nonsmooth objectives, continuously differentiable inequality constraints, and affine equality constraints. This leads to new conditions ensuring local minimizers, rather than merely stationary points, for such problems, extending beyond current state-of-the-art results that typically require twice continuously differentiable objectives and constraints.



\end{abstract}

\tableofcontents
\section{Introduction}
Variational convexity, introduced by Rockafellar in \cite{rock-vietnam}, states that the subdifferential of a nonconvex function is locally indistinguishable, in a primal-dual sense, from that of a convex surrogate. This concept has far-reaching implications in optimization, most notably by ensuring local optimality even without convexity. For this reason, variational convexity and its strong variant, variational strong convexity, have been investigated to establish local optimality in nonlinear programming \cite{Khanh2023variational,khanh2025second,wang2023strong,mordukhovich2026characterizations}, local saddle points of augmented Lagrangians \cite{rockafellar2023augmented}, and convergence of the augmented Lagrangian method \cite{rockafellar2023convergence,wang2023strong}.

In this paper, we develop new characterizations of variational convexity, establish calculus rules for this notion, and apply these results to derive new local optimality conditions in nonlinear programming.

First, we characterize variational convexity in terms of proximal hulls of the underlying function. Similar to Moreau envelopes, proximal hulls are regularizations of functions that lie between the original function and its Moreau envelopes and are always hypoconvex. Rockafellar characterized in \cite[Theorem 1]{Rockafellar2024} (see also \cite[Theorem 1]{rock-vietnam}) the variational convexity of a proper function \(\func{f}{\Rn}{\RE\defeq[-\infty,+\infty]}\) at \(\bar x\in\dom f\) for \(\bar v\in\lsubdiff f(\bar x)\) via a localized convex subgradient inequality for an \(f\)-attentive-\(\varepsilon\)-truncation of \(\lsubdiff f\).  While \(f\)-attentive restriction is generally unavoidable, it can be removed under subdifferential continuity of \(f\) at \(\bar x\) for \(\bar v\in\lsubdiff f(\bar x)\). We show that variational convexity of \(f\) is equivalent to prox-regularity of \(f\) together with a localized convex subgradient inequality for its proximal hull, without invoking \(f\)-attentive truncations. This result also complements the Moreau envelope characterization by Khanh, Mordukhovich, and Phat \cite[Theorem 3.2]{Khanh2023variational}, where variational convexity of \(f\) at \(\bar x\) for \(\bar v\in\lsubdiff f(\bar x)\)  is characterized by prox-regularity of \(f\) and local convexity of its Moreau envelope around the shifted point \(\bar x+\lambda\bar v\). Proximal hulls transfer variational convexity to a local property around \(\bar x\) itself, without requiring any shift. We also provide a geometric characterization of variational convexity by relating this notion to the epigraph of the underlying function. In analogy with classical convexity, variational convexity of a function can be characterized through the variational convexity of its epigraph.

Characterizations of variational convexity, including the new ones established in this paper, provide valuable insight into the nature of this notion. However, directly identifying functions with this property is often difficult. We address this fundamental challenge by developing calculus rules for variational convexity, namely, operations under which the property is preserved. Such rules make it possible to verify variational convexity by first identifying the property on simpler building blocks of a function and then checking standard constraint qualifications from variational analysis, which are often easier to verify. Elementary calculus rules were developed in \cite{luo2024level}, including a separable sum rule and a sum rule under restrictive assumptions. In this paper, we study calculus rules systematically and show that variational convexity is preserved under nonlinear compositions, summation, and proximal averaging. These results not only provide a convenient way to construct new classes of variationally convex functions, but also lead to new local optimality conditions for nonlinear programming problems.
	
We consider a constrained problem of the form
	\begin{align*}
		\min_{x\in\Rn}& \ f(x),\\
		\text{s.t.~}& \ g_i(x)\leq0,i=1,\ldots,m\\
		& \ g_i(x)=0,i=m+1,\ldots,p
	\end{align*}
where \(\func{f}{\Rn}{\RE}\) is a proper  and lower semicontinuous (lsc) function, and \(\func{g_i}{\Rn}{\R}\) is continuously differentiable for \(i=1,\ldots,m\), and affine for \(i=m+1,\ldots,p\). Variational convexity and its strong variant of \(f+\delta_\Omega\), where \(\Omega\) denotes the feasible set of the above problem, entail local optimality of the above constrained problem, which has attracted considerable interest recently; see, for instance, the works by Khanh, Khoa, Mordukhovich, and Phat \cite[Section 7]{Khanh2023variational}, \cite[Section 9]{khanh2025second}; see also recent work on nonlinear semidefinite programming by Wang, Ding, Yang, and Zhao \cite[Section 2]{wang2023strong}. While these works do not impose affine equality constraints, they require both the objective function and the constraint functions to be twice continuously differentiable. In contrast, in our approach it suffices for the constraint functions to be continuously differentiable, while the objective function may even be nonsmooth, at the cost of restricting equality constraints to be affine. Consequently, our framework yields local optimality conditions beyond the scope of the current literature.

The remainder of this paper is organized as follows.
In \cref{sec:pre}, we review preliminaries from nonsmooth and variational analysis, and variational convexity. \Cref{sec:chracterizations} presents characterizations of variational convexity. \Cref{sec:calculus} is devoted to the study of calculus rules for variational convexity, followed in \cref{sec:nonlinear} by applications to nonlinear programming.
	
\section{Preliminaries}\label{sec:pre}
\subsection{Elements of nonsmooth and variational analysis}

In this subsection we recall some preliminary concepts from nonsmooth and variational analysis. Our notation is standard and follows \cite{rockafellar_variational_1998}.

The function \(\func{f}{\Rn}{\RE}\) is said to be \emph{proper} if $f > -\infty$ and there exists \(x\in\Rn\) such that \(f(x)<+\infty\). Its \emph{domain} is defined by $\dom f := \{x \in \Rn \ | \ f(x) < +\infty\}$. Moreover, \(\func{f}{\Rn}{\RE}\) is called \emph{1-coercive} if \(\lim_{\norm{x}\to+\infty}f(x)/\norm{x}=+\infty\).

A set-valued operator \(\ffunc{T}{\Rn}{\Rn}\) is said to be \emph{monotone} if \(\ip{x-y}{u-v}\geq0\) for every \((x,u),(y,v)\in\gra T\), where $\gra T:=\{(z,w) \in \Rn \times \Rn \ | \ w \in T(z)\}$ denotes the \emph{graph} of $T$.
\begin{definition}
	A set-valued operator \(\ffunc{T}{\Rn}{\Rn}\) is said to be
	\begin{enumerate}
		\item \emph{outer semicontinuous} (osc) at \(\bar x \in \Rn\) if \(\set{\bar v\in\Rn}[\exists x_k \to\bar x, v_k \in T(x_k) \ \mbox{such that} \ v_k \to \bar v]\subseteq T(\bar x)\).
		\item\emph{locally bounded} at \(\bar x \in \Rn\) if there exists \(\varepsilon>0\) such that \(\bigcup_{\norm{x-\bar x}<\varepsilon}T(x)\) is bounded.
		\item \emph{upper semicontinuous} (usc) at \(\bar x \in \dom f\) if for every open set $O \subseteq \R^n$ with \(T(\bar x) \subseteq O\) there exists \(\varepsilon>0\) such that \(\bigcup_{\norm{x-\bar x}<\varepsilon}T(x)\subseteq O\).
	\end{enumerate}
\end{definition}
The \emph{Moreau envelope} of \(\func{f}{\Rn}{\RE}\) with parameter \(\lambda>0\) is the function
\[
	\func{\env f}{\Rn}{\RE}, \quad \env f(x) := \inf_{w\in\Rn}\set{f(w)+\tfrac{1}{2\lambda}\norm{w-x}^2}.
\]
The \emph{proximal operator} of \(\func{f}{\Rn}{\RE}\) with parameter \(\lambda>0\) is the set-valued operator
\[\ffunc{\prox f}{\Rn}{\Rn},\quad {\prox f}(x) := \argmin_{w\in\Rn}\set{f(w)+\tfrac{1}{2\lambda}\norm{w-x}^2}.
\]
The \emph{proximal hull} of \(\func{f}{\Rn}{\RE}\) with parameter \(\lambda>0\) is the function
\[
	\func{\hull f(x)}{\Rn}{\RE}, \quad \hull f(x) :=
	\sup_{w\in\Rn}\set{\env f(w)-\tfrac{1}{2\lambda}\norm{w-x}^2}.
\]
It holds that
\begin{equation}\label{eq:hull vs env}
		(\forall \lambda>0)
	\quad
	\env f\leq \hull f\leq f,
\end{equation}
see \cite[Example 1.44]{rockafellar_variational_1998}. Moreover, the Moreau envelope admits the representation
\[
	\lambda\env f=\tfrac{1}{2}\norm{\cdot}^2-(\lambda f+\tfrac{1}{2}\norm{\cdot}^2)^*,
\]
see \cite[Example 11.26]{rockafellar_variational_1998}. Here, for a function 
\(\func{g}{\Rn}{\RE}\), its \emph{conjugate function}  
\(\func{g^*}{\Rn}{\RE}\) is defined by $g^*(p) :=
    \sup_{x\in\Rn}
    \bigl\{p^\top x - g(x)\bigr\}$. The proximal hull admits the representation  
\[
	\lambda \hull f=(\lambda f+\tfrac{1}{2}\norm{\cdot}^2)^{**}-\tfrac{1}{2}\norm{\cdot}^2,
\]
which means that \(\hull f\) is \(\lambda\)-hypoconvex, that is, \(\hull f+\tfrac{1}{2\lambda}\norm{\cdot}^2\) is convex, see \cite[Example 11.26]{rockafellar_variational_1998}.

If \(f\) is proper, convex and lower semicontinuous, then the Moreau envelope \(\env f\) is finite-valued, the proximal operator \(\prox f\) is single-valued, and the proximal hull satisfies \(\hull f=f\) for every \(\lambda>0\).
In general, the behavior of these objects is governed by the following notion from variational analysis.
\begin{definition}
	A function \(\func{f}{\Rn}{\RE}\) is said to be \emph{prox-bounded} if there exists \(\lambda>0\) such that \(\env f(x)>-\infty\) for some \(x\in\Rn\).
	The supremum of all such \(\lambda >0\) is called the \emph{threshold of prox-boundedness}, denoted by \(\lambda_f\).
\end{definition}
In particular, any function that admits an affine minorant, such as any proper, convex, and lower semicontinuous function, is prox-bounded with the threshold being \(+\infty\).\ If \(f\) is prox-bounded with threshold \(\lambda_f\) and \(0<\lambda<\lambda_f\), then the proximal operator \(\prox f\) is nonempty everywhere, compact-valued, and outer semicontinuous, and the Moreau envelope \(\env f\) is always finite and continuous \cite[Theorem 1.25]{rockafellar_variational_1998}, whereas the proximal hull \(\hull f\) is proper \cite[Example 1.44]{rockafellar_variational_1998}.
\begin{definition}
Let \(\func{f}{\Rn}{\RE}\) be a proper function, and let \(\bar x\in\dom f\).
\begin{enumerate}
	\item The \emph{proximal subdifferential} of \(f\) at \(\bar x\) is defined by
	\[
	\partial_pf(\bar x)
	:=
	\set{\bar v\in\Rn}[
		% \begin{aligned}
			% &
            \exists r,\varepsilon>0
		\text{ such that }
        % \\
			% &
            f(x)\geq f(\bar x)+\ip{\bar v}{x-\bar x}-\tfrac{r}{2}\norm{x-\bar x}^2 \ \forall \norm{x-\bar x}\leq\varepsilon
		% \end{aligned}
	].
	\]
	\item The \emph{Fr\'echet/regular subdifferential} of \(f\) at \(\bar x\) is defined by
	\[
		\rsubdiff f(\bar x)
		:=
		\set{\bar v\in\Rn}[
				\liminf_{x\to\bar x}\tfrac{f(x)-f(\bar x)-\ip{\bar v}{x-\bar x}}{\norm{x-\bar x}}\geq0
		].
	\]
	\item The \emph{limiting/general/Mordukhovich subdifferential} of \(f\) at \(\bar x\) is defined by
	\[
		\lsubdiff f(\bar x)
		:=
		\set{\bar v\in\Rn}[
			\exists x_k\xrightarrow[]{f}\bar x, v_k\in\rsubdiff f(x_k) \text{ such that } v_k\to \bar v
		],
	\]
	where \(x_k\xrightarrow[]{f}\bar x\) means that \(x_k\to\bar x\) with \(f(x_k)\to f(\bar x)\) as $k \to +\infty$.
	\item The \emph{horizon/singular subdifferential} of \(f\) at \(\bar x\) is defined by
	\[
		\partial^\infty f(\bar x)
		:=
		\set{\bar v\in\Rn}[
			\exists \lambda_k\downarrow 0, x_k\xrightarrow{f}\bar x, v_k\in\rsubdiff f(x_k)
			\text{ such that }
			\lambda_kv_k\to\bar v
		].
	\]
\end{enumerate}
\end{definition}
\begin{definition}\label{def:f truncation}
	Let \(\func{f}{\Rn}{\RE}\) be a proper function and \(\varepsilon>0\).
	An \emph{\(f\)-attentive-\(\varepsilon\)-truncation} of the subdifferential mapping \(\ffunc{\lsubdiff f}{\Rn}{\Rn}\) around \((\bar x,\bar v)\in\gra\lsubdiff f\) is the mapping \(\ffunc{T}{\Rn}{\Rn}\) defined by
	\begin{equation*}
			T(x)
	=
		\begin{cases}
			\set{v\in\lsubdiff f(x)}[\norm{v-\bar v}<\varepsilon], &\text{ if }\norm{x-\bar x}<\varepsilon,f(x)<f(\bar x)+\varepsilon,\\
			\emptyset,&\text{ otherwise}.
		\end{cases}
	\end{equation*}
	For simplicity, we shall refer to \(T\) as an \emph{\(f\)-attentive truncation} whenever there is no ambiguity in omitting \(\varepsilon\).
\end{definition}
\begin{definition}
	A proper and lower semicontinuous function \(\func{f}{\Rn}{\RE}\) is said to be \emph{prox-regular} at \(\bar x\in\dom f\) for \(\bar v\in\lsubdiff f(\bar x)\) if there exist \(\varepsilon>0\) and \(r\geq0\) such that for every \((x,v)\in\gra\lsubdiff f\), satisfying \(\norm{x-\bar x}<\varepsilon\), \(f(x)<f(\bar x)+\varepsilon\), and \(\norm{v-\bar v}<\varepsilon\), it holds 
		\begin{equation*}
			(\forall \norm{x'-\bar x}<\varepsilon)\quad
			f(x')\geq f(x)+\ip{v}{x'-x}-\tfrac{r}{2}\norm{x'-x}^2.
		\end{equation*}
\end{definition}
Prox-regularity yields remarkable local regularity properties. In particular, when \(f\) is proper, lower semicontinuous, and prox-bounded, Poliquin and Rockafellar
\cite[Theorem 4.4]{poliquin1996} showed that prox-regularity of \(f\) at \(\bar x\) for
\(\bar v\in\lsubdiff f(\bar x)\) implies that, for every sufficiently small \(\lambda>0\), there exists an open neighborhood $U_\lambda$ of \(\bar x+\lambda \bar v\) on which the proximal mapping \(\prox f\) is single-valued, satisfies
\(\prox f(\bar x+\lambda \bar v)=\{\bar x\}\), and the Moreau envelope \(\env f\) is differentiable with Lipschitz continuous gradient given by
\[
    (\forall u\in U_\lambda)\qquad
    \nabla \env f(u)=\frac{u-\prox f(u)}{\lambda}.
\]
See also \cite[Proposition 13.37]{rockafellar_variational_1998}. This local smoothness property of prox-regular functions will play a central role in our analysis.

\subsection{Variational convexity}\label{sec:vc intro}
Prox-regularity of \(f\) is equivalent to the hypomonotonicity of an \(f\)-attentive truncation \(T\) of \(\lsubdiff f\) \cite[Theorem 13.36]{rockafellar_variational_1998}. This naturally raises the question of which property of \(f\) corresponds to monotonicity of such a truncation.

Rockafellar \cite{rock-vietnam} answered this question through the introduction of \emph{variational convexity}, the central topic of this paper. As a strengthening of prox-regularity, variational convexity yields significantly stronger consequences and has recently attracted considerable attention. We first recall the definition.

\begin{definition}\cite[Definition 2]{rock-vietnam}\label{def:vc}
Let \(\func{f}{\Rn}{\RE}\) be proper and lower semicontinuous. We say that \(f\) is \emph{variationally convex} at \(\bar x\in\dom f\) for \(\bar v\in\lsubdiff f(\bar x)\) if there exist open convex neighborhoods $U$ of $\bar x$ and $V$ of $\bar v$, a convex and lower semicontinuous function \(\func{\hat f}{\Rn}{\RE}\) satisfying \(\hat f\leq f\) on \(U\), and a scalar $\varepsilon>0$ such that
	\begin{equation}\label{def:s-vc}
		\left[
			U_\varepsilon\times V
		\right]
		\cap\gra\lsubdiff f
		=
		\left[
			U\times V
		\right]
		\cap\gra\lsubdiff\hat f,
	\end{equation}
where \(
		U_\varepsilon\defeq\set{x\in U}[f(x)< f(\bar x)+\varepsilon]
	\) 
    is an \(f\)-attentive neighborhood of \(\bar x\). Moreover, $f$ and $\hat f$ coincide at every point appearing in the common graph \eqref{def:s-vc}. If, in addition, \(\hat f\) can be chosen to be strongly convex, then \(f\) is said to be variationally strongly convex at \(\bar x\) for \(\bar v\in\lsubdiff f(\bar x)\)
\end{definition}
\begin{remark}
Some comments are in order:
\begin{enumerateq}
			\item\label{rem: subdiff cts}
			The \(f\)-attentive neighborhood \(U_\varepsilon\) in \eqref{def:s-vc} can be replaced by \(U\) whenever \(f\) is \emph{subdifferentially continuous} at \(\bar x \in \dom f\) for \(\bar v\in\lsubdiff f(\bar x)\); that is,
			\(
			f(x_k)\to f(\bar x)
			\)
			for every sequence \((x_k,v_k)\to(\bar x,\bar v)\) as $k \to +\infty$ with \(v_k\in\lsubdiff f(x_k)\); see \cite[Definition 13.28]{rockafellar_variational_1998}. In particular, every proper, convex and lower semicontinuous function is subdifferentially continuous. Indeed, let  \(\func{f}{\Rn}{\RE}\) be proper, convex and lower semicontinuous, and let \((x_k, v_k) \to (\bar x, \bar v) \) as $k \to +\infty$ with \(v_k\in\lsubdiff f(x_k)\). By the convex subgradient inequality,
			\[
			\limsup_{k\to+\infty}f(x_k)
			\leq
			\limsup_{k\to+\infty} \left(f(\bar x)-\ip{\bar x-x_k}{v_k} \right)
			=
			f(\bar x),
			\]
			yielding, in combination with the lower semicontinuity of $f$, \(f(x_k)\to f(\bar x)\) as $k \to +\infty$. Consequently, every proper, convex and lower semicontinuous function is variationally convex at every point of its domain for every subgradient.
			\item\label{rem:vc and prox-regular}
			Variational convexity of \(f\) at \(\bar x \in \dom f\) for \(\bar v \in \partial f(\bar x)\) implies prox-regularity at \(\bar x\) for \(\bar v\).\
			Indeed, let \((x,v)\) belong to the truncation in \eqref{def:s-vc}. The subgradient inequality for \(\hat f\) yields
			\[
			(\forall x'\in U)\quad
			f(x')\geq \hat f(x')\geq \hat f(x)+\ip{v}{x'-x}=f(x)+\ip{v}{x'-x},
			\]
			therefore $\hat f$ is prox-regular at $\bar x$ for $\bar v$. We refer the reader to \cite{Rockafellar2024} for a systematic study of the relationship between variational convexity and prox-regularity. In particular, Rockafellar introduced the notion of variational $s$-convexity, for $s \in \R$, as a unified framework encompassing both concepts.
		\end{enumerateq}
	\end{remark}

In his seminal paper \cite{rock-vietnam}, Rockafellar motivated variational convexity through the search for a local counterpart of the classical characterization of convexity: a proper and lower semicontinuous function is convex if and only if its subgradient mapping is (maximally) monotone; see, for example, \cite[Theorem 12.17]{rockafellar_variational_1998}. A natural first approach to establishing such a local correspondence is to investigate the relationship between local convexity of a function and monotonicity of its subgradient mapping restricted to neighborhoods. However, monotonicity of \(f\)-attentive truncations of the subgradient mapping is a substantially more general notion. Unlike monotonicity on neighborhoods, this localized form of monotonicity may arise even when the underlying function fails to be locally convex; see \cite[Section~1]{rock-vietnam} for illustrative examples. Rockafellar proved that monotonicity of an \(f\)-attentive truncation is in fact equivalent to variational convexity \cite[Theorem~1]{rock-vietnam}, thereby establishing the desired local analogue of the classical convexity--monotonicity correspondence.

The distinction between variational convexity and local convexity extends beyond the characterization of localized monotonicity. Indeed, for a proper and lower semicontinuous function \(\func{f}{\Rn}{\RE}\), one has
\begin{equation}\label{eq:cvx equivalence1}
	f \text{ is convex }
	\quad\Longleftrightarrow\quad
	(\exists \lambda>0)\;\prox f
	\text{ is firmly nonexpansive},
\end{equation}
see \cite[Theorem~1]{rock2021}. This naturally raises the question of how local firm nonexpansiveness can be characterized. A recent result shows that local firm nonexpansiveness of the proximal mapping is characterized by variational convexity of \(f\), rather than by local convexity; see \cite[Theorem~4.7]{luo2024level}; see also \cite[Theorem 2]{rock2021} for a characterization via localized proximal operators. This provides further evidence that variational convexity is the appropriate local analogue of convexity in variational analysis. For a broader perspective on the role of variational convexity and its connections with other localized notions arising in convex optimization, we refer the reader to \cite[Figure~1]{luo2024level}.

Variational convexity is particularly appealing because it yields local optimality without requiring local convexity of the underlying function. To see this, suppose that \(f\) is variationally convex at \(\bar x \in \dom f\) for \(0\in\lsubdiff f(\bar x)\). Then \(\bar x\) is a local minimizer of \(f\). Indeed, \eqref{def:s-vc} implies that
\[
0\in\lsubdiff \hat f(\bar x),
\]
and hence, by convexity of \(\hat f\),
\begin{equation}\label{eq:local min}
	(\forall x\in U)\qquad
	f(x)\ge \hat f(x)\ge \hat f(\bar x)=f(\bar x).
\end{equation}
Rockafellar \cite{rock-vietnam} referred to this property as \emph{variational sufficiency}. Variational sufficiency, together with its strong variant, has attracted considerable recent attention; see, for example,
\cite{Khanh2023variational,rockafellar2023augmented,wang2023strong,khanh2025second}.
We postpone a more detailed discussion of this line of research to \cref{sec:nonlinear}. Recent developments concerning variational convexity have also focused on the modulus of convexity
\cite{rockafellar2025derivative,adly2026variational}
and the stability of local solution mappings
\cite{benko2024primal}.

Several useful characterizations of variational convexity are summarized below. They will play a key role throughout the remainder of the paper.

\begin{lemma}\label{thm:variational cvx}
% \todo{Make it less a formal lemma}
	Let \(\func{f}{\Rn}{\RE}\) be proper and lower semicontinuous function, let \(\bar x\in\dom f\) and \(\bar v\in\lsubdiff f(\bar x)\).
	Then the following statements are equivalent:
	\begin{enumerateq}
		\item\label{thm:variational cvx:1}
		\(f\) is variationally convex at \(\bar x\) for \(\bar v\).
		\item\label{thm:variational cvx:2}
		\emph{\cite[Theorem 1]{Rockafellar2024}}
		\(\bar v\in\rsubdiff f(\bar x)\) and
		\(\partial f\) is \(f\)-locally monotone around \((\bar x,\bar v)\); that is, there exist open convex neighborhoods $U$ of $\bar x$ and $V$ of $\bar v$, and \(\varepsilon>0\) such that
		\begin{equation}
				\gph\partial f\cap[U_\varepsilon\times V] \text{ is monotone}.
		\end{equation}
		\item\label{thm:variational cvx:3}
		\emph{\cite[Theorem 3.2]{Khanh2023variational}}
		\(f\) is prox-regular at \(\bar x\) for \(\bar v\), and the Moreau envelope \(\env f\) is locally convex around \(\bar x+\lambda\bar v\) for all sufficiently small \(\lambda\).
		\item\label{thm:variational cvx:4}
		\emph{\cite[Theorem 1]{Rockafellar2024}}
		There exist open convex neighborhoods $U$ of $\bar x$ and $V$ of $\bar v$, and \(\varepsilon>0\) such that, for every \((x,v)\in \gra\partial f\cap[U_\varepsilon\times V]\),
		\begin{equation}\label{eq:variational cvx:4}
		(\forall y\in U)\quad
		f(y)\geq f(x)+\ip{v}{y-x},
		\end{equation}
		% whenever \((x,v)\in \gra\partial f\cap[U_\varepsilon\times V]\),
		\end{enumerateq}
		where \(U_\varepsilon=\set{x\in U}[f(x)<f(\bar x)+\varepsilon]\).
\end{lemma}
\begin{remark}[known characterizations]
Variational convexity admits several other characterizations, including those in terms of proximal operators \cite[Theorem 4.7]{luo2024level}, second-order subdifferentials \cite[Theorem 5.4]{Khanh2023variational, khanh2025second}, quadratic bundles \cite[Theorem 4.4]{khanh2025bundle}, and graphical derivatives \cite[Theorem 4.6]{garrido2025characterization}. See also \cite{khanh-phat} for infinite-dimensional results. Here we collect only those characterizations that will be used in the derivation of our main results.
\end{remark}
Next we provide a useful example of a variationally convex function. We refer the reader to \cite[Section 1]{rock-vietnam} and \cite[Section 2]{Khanh2023variational} for other elementary examples. Further examples will be presented throughout the remainder of this paper as we establish our main results.

For each \(x\in\Rn\), define the index set of nonzero components by
	\[
		\mathcal I(x)=\set{i\in\set{1,\ldots,n}}[x_i\neq0].
	\]
The zero norm of \(x\in\Rn\) is defined by
\[
	\norm{x}_0=\card(\mathcal I(x)),
\]
where \(\card(I)\) denotes the cardinality of a set \(I\).
It is well known that	
\begin{equation}\label{eq:zero norm subdiff}
		(\forall x\in\Rn)\quad
		\rsubdiff\norm{\cdot}_0(x)=\lsubdiff\norm{\cdot}_0(x)=\set{v\in\Rn}[(\forall i\in \mathcal I(x))~v_i=0],
	\end{equation}
see, for example, \cite[Theorem 1 and Remark 2]{le2013generalized}.

Variational convexity of \(\norm{\cdot}_0\) at \(\bar x=0\) was established in \cite[Equation (7)]{rock-vietnam} for the one-dimensional case and in \cite[Example 2.5]{Khanh2023variational} for the \(n\)-dimensional case. To the best of our knowledge, a proof for an arbitrary nonzero point \(\bar x\in\Rn\) has not appeared in the literature. We therefore include one below for completeness.
\begin{example}\label{eg:zero norm vc}
	The zero norm \(\norm{\cdot}_0\) is variationally convex at every \(\bar x\in\Rn\) for every \(\bar v\in\lsubdiff \norm{\cdot}_0 (\bar x)\).
\end{example}
\begin{proof}
Let \(\bar x\in\Rn\) and \(\bar v\in\lsubdiff \norm{\cdot}_0(\bar x)\). Choose \(0<\varepsilon<1\).
	Then
	\[
	(\forall x\in U_\varepsilon)\quad
		\norm{x}_0\leq\norm{\bar x}_0,
	\]
	where \(U_\varepsilon:=\set{x\in \bar x+\varepsilon\ball}[\norm{x}_0 < \norm{\bar x}_0+\varepsilon]\). Shrinking \(\varepsilon>0\) if necessary, we may assume without loss of generality that
	\(
	(\forall x\in \bar x+\varepsilon\ball)~
	\mathcal I(\bar x)\subseteq \mathcal I(x)
	\). Therefore
	\[
		(\forall x\in U_\varepsilon)\quad
		\norm{\bar x}_0\leq\norm{x}_0\leq\norm{\bar x}_0,
	\]
	entailing that
	\begin{equation}\label{eq:zero norm property}
		(\forall x\in U_\varepsilon)\quad
		\mathcal I(x)=\mathcal I(\bar x)
		\text{ and }
		\lsubdiff\norm{\cdot}_0(x)=\set{v\in\Rn}[(\forall i\in \mathcal I(\bar x))~v_i=0],
	\end{equation}
	where the subdifferential equality follows from \eqref{eq:zero norm subdiff}.
	Now take \((x,v), (y,u)\in\gra\lsubdiff f\cap[U_\varepsilon\times \Rn]\).
	Then
	\begin{equation*}
		\ip{x-y}{v-u}=\sum_{i=1}^n\ip{x_i-y_i}{v_i-u_i}=\sum_{i\in \mathcal I(\bar x)}\ip{x_i-y_i}{v_i-u_i}=0,
	\end{equation*}
	where the second and the last equalities hold because of \eqref{eq:zero norm property}.
	Since
\(\lsubdiff \norm{\cdot}_0(\bar x)=\rsubdiff\norm{\cdot}_0(\bar x)\) by \eqref{eq:zero norm subdiff},
we have
\(\bar v\in\rsubdiff\norm{\cdot}_0(\bar x)\).
Consequently, all the conditions of \cref{thm:variational cvx:2} hold, establishing the desired variational convexity.
\end{proof}

\section{New characterizations of variational convexity}\label{sec:chracterizations}
\subsection{A characterization via the proximal hull}

Rockafellar's characterization of variational convexity \cite[Theorem 1]{Rockafellar2024} states that \(f\) is variationally convex at \(\bar x \in \dom f\) for \(\bar v\in\lsubdiff f(\bar x)\) if and only if there exist open convex neighborhoods $U$ of $\bar x$ and $V$ of $\bar v$, and a constant \(\varepsilon>0\) such that, for every \((x,v)\in \gra\partial f\cap[U_\varepsilon\times V]\),
		\begin{equation*}
		(\forall y\in U)\quad
		f(y)\geq f(x)+\ip{v}{y-x},
		\end{equation*}
where \(U_\varepsilon=\set{x\in U}[f(x)<f(\bar x)+\varepsilon]\). The appearance of the \(f\)-attentive neighborhood \(U_\varepsilon\) obscures the precise region on which this inequality is valid. When \(f\) is subdifferentially continuous at \(\bar x\) for \(\bar v\), the \(f\)-attentive restriction can be dropped; see \cref{rem: subdiff cts}.
The question therefore arises whether such restrictions can be eliminated in general.

In this section, we show that variational convexity of \(f\) is equivalent to the combination of prox-regularity of \(f\) and a localized convex subgradient inequality for \(\hull f\), without any appeal to attentive neighborhoods.
The following lemma will be instrumental in establishing this equivalence.
\begin{lemma}\label{lem:prox hull subdiff}
	Let \(\func{f}{\Rn}{\RE}\) be a proper, lower semicontinuous, and prox-bounded function with threshold \(\lambda_f>0\).
	Then
	\[
	(\forall 0<\lambda<\lambda_f)\quad
		v\in\lsubdiff\hull f(x)
	\Leftrightarrow
		x\in\conv\left(\prox f(x+\lambda v)\right).
	\]
\end{lemma}
\begin{proof}
We write \(j\defeq\norm{\cdot}^2/2\) for simplicity. Since \(0<\lambda<\lambda_f\), \(\hull f=(f+\lambda^{-1}j)^{**}-\lambda^{-1}j\) is proper and \(1/\lambda\)-hypoconvex; see for instance \cite[Example 11.26]{rockafellar_variational_1998}.
	Then
	\begin{align*}
		u\in\prox\hull f(x)
	\Leftrightarrow
		0\in\lsubdiff\hull f(u)+\tfrac{1}{\lambda}(u-x),
	\end{align*}
	where the equivalence holds because the function \(w\mapsto\hull f(w)+\tfrac{1}{2\lambda}\norm{w-x}^2\) is convex.
	Hence
	\(
	\prox\hull f=\left(\Id+\lambda\lsubdiff\hull f\right)^{-1}
	\)
	and, consequently,
	\begin{align*}
		v\in\lsubdiff\hull f(x)
	\Leftrightarrow
		x\in(\Id+\lambda\lsubdiff\hull f)^{-1}(x+\lambda v)=\prox\hull f(x+\lambda v).
	\end{align*}
	A result by Chen, Wang, and Planiden \cite[Lemma 2.9]{chen2020proximal} states that \(\prox\hull f(x+\lambda v)=\conv\left(\prox f(x+\lambda v)\right)\), from which the claimed relation readily follows.
\end{proof}
Now we provide the main result of this section.
\begin{theorem}\label{thm:vc}
	Let \(\func{f}{\Rn}{\RE}\) be proper, lower semicontinuous and prox-bounded, let \(\bar x\in\dom f\) and \(\bar v\in\lsubdiff f(\bar x)\). Then the following statements are equivalent:
	\begin{enumerateq}
		\item\label{thm:vc:1} \(f\) is variationally convex at \(\bar x\) for \(\bar v\in\lsubdiff f(\bar x)\).
		% \item\label{thm:vc:2} \(\bar v\in\lsubdiff_p f(\bar x)\) and for \(\lambda>0\) sufficiently small the resolvent \((\Id+\lambda\lsubdiff f)^{-1}\) is firmly nonexpansive around \(\bar x+\lambda\bar v\) in the following sense:
		% there exists \(\varepsilon>0\) such that for \(\norm{u_i-(\bar x+\lambda\bar v)}<\varepsilon\) and \(x_i\in(\Id+\lambda\lsubdiff f)^{-1}(u_i)\) satisfying
		% \begin{equation}\label{eq:resolvent fnp:condition}
		% \norm{x_i-\bar x}<\varepsilon,
		% f(x_i)<f(\bar x)+\varepsilon,
		% \end{equation}
		% it holds that
		% \begin{equation}\label{eq:resolvent fnp}
		% 	\ip{x_1-x_2}{u_1-u_2}\geq\norm{x_1-x_2}^2.
		% \end{equation}
		\item\label{thm:vc:5}
		\(f\) is prox-regular at \(\bar x\) for \(\bar v\) and for sufficiently small \(\lambda>0\) there exist open convex neighborhoods $U$ of $\bar x$ and $V$ of $\bar v$ such that, for every \((x,v)\in\gra\lsubdiff\hull f\cap[U\times V]\),
		\begin{equation}\label{eq:hull vc ineq}
			(\forall z\in U)\quad
			\hull f(z)\geq \hull f(x)+\ip{v}{z-x}.
		\end{equation}
	\end{enumerateq}
\end{theorem}
\begin{proofitemize}
	\item ``\ref{thm:vc:1} \(\Rightarrow\) \ref{thm:vc:5}''
    This is the nontrivial implication. Variational convexity implies that \(f\) is prox-regular at \(\bar x\) for \(\bar v\). For sufficiently small \(0<\lambda<\lambda_f\), we have that, through \cref{thm:variational cvx:3},
	\begin{equation}\label{eq:single-valued prox}
	 \prox f\text{ is single-valued and } \env f\text{ is convex on \(U_\lambda\defeq\bar x+\lambda\bar v+\varepsilon_1\ball\) for some \(\varepsilon_1>0\)},
	\end{equation}
	see \cite[Proposition 13.37]{rockafellar_variational_1998}.\
	We claim that there exists \(0 < \varepsilon < \varepsilon_1\) such that, for every $u \in \bar x+\lambda\bar v+\varepsilon \ball$,
	\begin{equation}\label{eq:hull ineq}
		(\forall \norm{z-\bar x}<\varepsilon)~
		\hull f(z)
	\geq
		\hull f(\prox f(u))+\tfrac{1}{\lambda}\ip{u-\prox f(u)}{z-\prox f(u)}.
	\end{equation}
	Taking the claim \eqref{eq:hull ineq} as granted for now, we note that, for every $x \in \bar x + \tfrac{\varepsilon_1}{2} \ball$ and every $v \in \bar v + \tfrac{\varepsilon_1}{2\lambda} \ball$, it holds
	\begin{equation}\label{eq:restrictions}
		\norm{x+\lambda v-\bar x-\lambda\bar v}<\varepsilon_1.
	\end{equation}
Choose \(\hat\varepsilon<\min(\tfrac{\varepsilon_1}{2},\tfrac{\varepsilon_1}{2\lambda},\varepsilon)\), and
	\[
		U:=\bar x+\hat\varepsilon\ball\text{ and }V:=\bar v+\hat\varepsilon\ball.
	\]
	Pick \((x,v)\in\gra\lsubdiff\hull f\cap[U\times V]\).
	Then \cref{lem:prox hull subdiff} entails that
	\begin{equation}\label{eq:localization}
			% \norm{x+\lambda v-\bar x-\lambda\bar v}<\varepsilon
			% \text{ and consequently }
			x\in\conv\left(\prox f(x+\lambda v)\right)=\prox f(x+\lambda v),
	\end{equation}
	where the equality owes to \eqref{eq:restrictions} and \eqref{eq:single-valued prox}.
	Invoke \eqref{eq:restrictions} to see that our claim \eqref{eq:hull ineq} is applicable to \(u:=x + \lambda v\) and yields
	\begin{align*}
		(\forall z\in U)\quad
		\hull f(z)
		% &\geq
		% 	\hull f(\prox f(x+\lambda v))+\tfrac{1}{\lambda}\ip{x+\lambda v-\prox f(x+\lambda v)}{z-\prox f(x+\lambda v)}\\
		\geq
			\hull f(x)+\tfrac{1}{\lambda}\ip{x+\lambda v-x}{z-x}
			= \hull f(x)+\ip{v}{z-x},
	\end{align*}
	which is the sought \eqref{eq:hull vc ineq}.

Now we justify the claim \eqref{eq:hull ineq}. First we prove that there exists \(0 < \varepsilon_2 < \varepsilon_1\) such that, for every $u \in \bar x+\lambda\bar v + \varepsilon_2 \ball$,
	\begin{equation}\label{eq:infconv claim}
		(\forall z\in \bar x+\varepsilon_2\ball)\quad
		z+u-\prox f(u)\in U_\lambda.
	\end{equation}
Arguing by contradiction, suppose that there exist sequences \(u_k\to\bar x+\lambda\bar v\) and \(z_k\to\bar x\) as $k \to +\infty$ such that
	\[
	z_k+u_k-\prox f(u_k)\not\in U_\lambda \quad \forall k \in \N.
	\]
But \(\bar x=\prox f(\bar x+\lambda\bar v)\) owing to the prox-regularity of \(f\) at \(\bar x\) for \(\bar v\); see, e.g., \cite[Proposition 13.37]{rockafellar_variational_1998}. Since \(\prox f(u_k)\to \prox f(\bar x+\lambda\bar v) = \bar x\) and \(U_\lambda\) is open, we have
	\[
		\bar x+\lambda\bar v
	=
		\lim_{k\to+\infty}z_k+u_k-\prox f(u_k)
	\not\in U_\lambda,
	\]
which is absurd.

The argument that follows is inspired by the proof of \cite[Equation (2.31)]{Rockafellar2024}, but starts from a different property. Choose \(0 < \varepsilon< \varepsilon_2 \), and fix \(u\in\bar x+\lambda\bar v+\varepsilon\ball\).
	Define
	\[(\forall y\in\Rn)\quad
	q_u(y)
	:=
		\tfrac{1}{2}\norm{y-u}^2
	+
		\ip{\prox f(u)}{y-u}
	+
		(\lambda f+j)^*(u),
	\]
	where \((\lambda f+j)^*(u)\) is finite because \(f\) is prox-bounded with threshold \(\lambda_f>0\) and
	\[
	(\forall \lambda<\lambda'<\lambda_f)\quad
	f+\lambda^{-1}j=\underbracket{f+(\lambda')^{-1}j}_{\text{bounded below}}+(\lambda^{-1}-(\lambda')^{-1})j
	\text{ is 1-coercive.}
	\]
The convexity of \(\env f\)  on \(U_\lambda\) yields
	\begin{equation}\label{eq:cvx env}
		(\forall y\in U_\lambda)\quad
		\env f(y)\geq \env f(u)+\ip{\nabla \env f(u)}{y-u},
	\end{equation}
	where \(\lambda\nabla \env f(u)=u-\prox f(u)\) by prox-regularity \cite[Proposition 13.37]{rockafellar_variational_1998}.\
	Using that
	\(
		\lambda\env f=j-(\lambda f+j)^*,
	\)
\eqref{eq:cvx env} can be equivalently written as
	\begin{equation}\label{eq:conj0}
	(\forall y\in U_\lambda)\quad
		(\lambda f+j)^*(y)
	\leq
		(\lambda f+j)^*(u)
	+\ip{\prox f(u)}{y-u}
	+\tfrac{1}{2}\norm{y-u}^2.
	\end{equation}
	% Fix \(u\in U_\lambda\).
	We globalize \eqref{eq:conj0} by adding an indicator function
	\begin{equation}\label{eq:conj1}
	(\forall y\in\Rn)\quad
	(\lambda f+j)^*(y)
	\leq
	q_u(y)+\delta_{U_\lambda}(y).
	\end{equation}
	Taking conjugate on both sides of \eqref{eq:conj1} yields
	\begin{equation}\label{eq:infconv global}
	(\forall z\in\Rn)\quad
		(\lambda f+j)^{**}(z)
	\geq
		\left(
			q_u+\delta_{U_\lambda}
		\right)^*(z)
	=
		(q_u^*\infconv\sigma_{U_\lambda})(z),
	\end{equation}
	where \(q_u^*\infconv\sigma_{U_\lambda}\) denotes the infimal convolution of \(q_u^*\) and \(\sigma_{U_\lambda}\).\
	For every \(z\in\Rn\),
	\begin{align}
		q_u^*(z)
	&=
		\left(
			\tfrac{1}{2}\norm{\cdot-u}^2+\ip{\prox f(u)}{\cdot}
		\right)^*(z)
	+
		\ip{\prox f(u)}{u}-(\lambda f+j)^*(u)\nonumber\\
	&=
		\tfrac{1}{2}\norm{z-\prox f(u)}^2
	+
		\ip{u}{z}
	-
		(\lambda f+j)^*(u)\nonumber\\
	&=
		\tfrac{1}{2}\norm{z}^2
	+
		\tfrac{1}{2}\norm{\prox f(u)}^2
	+
		\ip{u-\prox f(u)}{z}
	+
		\lambda\env f(u)-\tfrac{1}{2}\norm{u}^2\nonumber\\
	&=
		\tfrac{1}{2}\norm{z}^2
	+
		\lambda f(\prox f(u))+\ip{u-\prox f(u)}{z-\prox f(u)},\label{eq:q conj}
	\end{align}
	where the last equality holds due to \(\lambda \env f(u)=\lambda f(\prox f(u))+\tfrac{1}{2}\norm{u-\prox f(u)}^2\), and hence
	\[
	(\forall z\in\Rn)\quad
	\nabla q_u^*(z)=z+u-\prox f(u).
	\]
From \eqref{eq:infconv claim} and our choice of \(\varepsilon\), we deduce that
	\begin{align*}
	(\forall z\in\bar x+\varepsilon\ball)\quad
		\nabla q_u^*(z)=z+u-\prox f(u)\in U_\lambda,
	\end{align*}
	and, consequently,
	\begin{align*}
		(\forall z\in\bar x+\varepsilon\ball)\quad
		0= N_{U_\lambda}\left(\nabla q_u^*(z)\right)=\left(\lsubdiff\sigma_{U_\lambda}\right)^{-1}\left(\nabla q_u^*(z)\right)
	\Leftrightarrow
		\nabla q_u^*(z)\in\lsubdiff\sigma_{U_\lambda}(0),
	\end{align*}
	which means that the minimum in the infimal convolution
	\(
		(q_u^*\infconv\sigma_{U_\lambda})(z)=\inf_{w\in\Rn}\set{q_u^*(w)+\sigma_{U_\lambda}(z-w)}
	\) is attained at \(w=z\).
	Hence
	\[	(\forall z\in\bar x+\varepsilon\ball)\quad
		(q_u^*\infconv\sigma_{U_\lambda})(z)
	=
		q_u^*(z)+\sigma_{U_\lambda}(0)
	=
		q_u^*(z).
	\]
	Appealing to \eqref{eq:infconv global} yields that
	\begin{equation}\label{eq:infconl attained}
		(\forall z\in\bar x+\varepsilon\ball)\quad
		(\lambda f+j)^{**}(z)
	\geq
		(q_u^*\infconv\sigma_{U_\lambda})(z)
	=
		q_u^*(z).
	\end{equation}
From here, \eqref{eq:hull ineq} can be justified by further expressing the above in terms of the proximal hull as
	\begin{align*}
	(\forall z\in\bar x+\varepsilon\ball)\quad
		\lambda\hull f(z)
	&=
		(\lambda f+j)^{**}(z)-\tfrac{1}{2}\norm{z}^2\geq q_u^*(z)-\tfrac{1}{2}\norm{z}^2\nonumber\\
	&=
		\lambda f(\prox f(u))+\ip{u-\prox f(u)}{z-\prox f(u)}\nonumber\\
	&=
		\lambda\hull f(\prox f(u))+\ip{u-\prox f(u)}{z-\prox f(u)},
	\end{align*}
	where the first equality follows from \cite[Example 11.26(c)]{rockafellar_variational_1998}, the inequality owes to \eqref{eq:infconl attained}, the second equality follows from \eqref{eq:q conj}, and the last equality holds because \(f=\hull f\) on \(\ran\prox f\); see \cite[Example 1.44]{rockafellar_variational_1998}.
	
\item ``\ref{thm:vc:5} \(\Rightarrow\) \ref{thm:vc:1}''
Since \(f\) is assumed to be prox-regular at \(\bar x\) for \(\bar v\), there exist $0 < \lambda < \lambda_f$, \(\varepsilon>0\) and an \(f\)-attentive-\(\varepsilon\)-localization \(T\) of \(\lsubdiff f\) such that
	\begin{equation}\label{eq:hull resolvent}
	\prox f=\left(\Id+\lambda T\right)^{-1}
	\text{ on }\bar x+\lambda\bar v+\varepsilon\ball,
	\end{equation}
see \cite[Proposition 13.37]{rockafellar_variational_1998}. We choose $0 < \lambda < \lambda_f$ such that there exist open convex neighborhoods $U$ of $\bar x$ and $V$ of $\bar v$ with the property that, for every \((x,v)\in\gra\lsubdiff\hull f\cap[U\times V]\), \eqref{eq:hull vc ineq} holds.

Now take \((x,v)\in\gra \lsubdiff f\) such that
	\[
	\norm{x-\bar x}<\frac{\varepsilon}{2}, \ f(x)<f(\bar x)+\varepsilon, \text{ and }
	\norm{v-\bar v}<\min\left(\frac{\varepsilon}{2\lambda},\varepsilon\right).
	\]
	Then \((x,v)\in\gra T\) and
	\(
		\norm{x+\lambda v-\bar x-\lambda\bar v}
	\leq
		\norm{x-\bar x}+\lambda\norm{v-\bar v}
	<\varepsilon.
	\)
	So
	\begin{equation*}
		x
	\in(\Id+\lambda T)^{-1}(x+\lambda v)=\prox f(x+\lambda v)\subseteq\conv\left(\prox f(x+\lambda v)\right),
	\end{equation*}
	where the equality owes to \eqref{eq:hull resolvent}. Invoking \cref{lem:prox hull subdiff}, it yields
	\(
		v
	\in
		\lsubdiff\hull f(x).
	\)
	Assume without loss of generality that
	\[
	\bar x+\frac{\varepsilon}{2}\ball\subseteq U\text{ and }
	\bar v+\min\left(\frac{\varepsilon}{2\lambda},\varepsilon\right)\ball\subseteq V,
	\]
it yields
	\begin{equation*}
		(x,v)\in\gra\lsubdiff\hull f\cap[U\times V].
	\end{equation*}
	Hence \eqref{eq:hull vc ineq} furnishes
	\begin{align*}
		(\forall z\in U)\quad
		f(z) \geq 
		\hull f(z)\geq \hull f(x)+\ip{v}{z-x} =
		f(x)+\ip{v}{z-x},
	\end{align*}
	where the equality owes to \(x\in\prox f(x+\lambda v)\) and the fact that \(f=\hull f\) on \(\ran\prox f\); see for instance \cite[Example 1.44]{rockafellar_variational_1998}.\
	The assertion of \ref{thm:vc:1} now follows readily from \cref{thm:variational cvx:4}.
	\qed
\end{proofitemize}
\begin{remark}
	\cref{thm:vc} complements a recent characterization of variational convexity via the Moreau envelope. Specifically, Khanh, Mordukhovich, and Phat \cite[Theorem 3.2]{Khanh2023variational} showed that variational convexity of \(f\) at \(\bar x\) for \(\bar v\in\lsubdiff f(\bar x)\) is equivalent to local convexity of the Moreau envelope \(\env f\) around the shifted point \(\bar x+\lambda\bar v\); see also \cite[Proposition 5.4]{poliquin1996}.\
	In contrast, \Cref{thm:vc} characterizes variational convexity through a local convex subgradient inequality for \(\hull f\) around \(\bar x\) itself, without any shift of the reference point.
\end{remark}

Finally, we present an example showing how \cref{thm:vc} reveals a hidden convex subgradient inequality for a nonconvex function of practical interest.
The Minimax Concave Penalty (MCP) with parameter \(\lambda>0\) is defined as
\begin{align*}
	\rho_\lambda(x)
= 
	\begin{cases}
		|x|-\tfrac{x^2}{2\lambda}, & \text{if } |x|\leq\lambda,\\
		\tfrac{\lambda}{2}, & \text{if } |x|>\lambda,
	\end{cases}
\end{align*}
see \cite[Section 2.1]{zhang2010nearly} or \cite[Equation (12)]{mazumder2011sparsenet}.
As a nonconvex approximation of the zero norm, MCP is widely used in statistical learning and signal processing; see e.g.\ \cite{mazumder2011sparsenet,li2024recursive} and the monograph \cite[Section 4.6]{hastie2015statistical}. Despite its nonconvexity, direct calculation shows that MCP satisfies the convex subgradient inequality described in \eqref{eq:hull vc ineq}. This is not merely a coincidence. Indeed, the MCP can be expressed as a rescaled proximal hull of the one-dimensional zero norm \(\norm{\cdot}_0\), namely,
\[
	\rho_\lambda(x) = \tfrac{\lambda}{2}h_\lambda \|\cdot\|_0\left(\sqrt{2\lambda^{-1}}x\right) \quad \forall x \in \R.
\]
To verify this representation, note that
\[
	h_\lambda\|\cdot\|_0(x) = 
	\begin{cases}
	\sqrt{2\lambda^{-1}}|x| - \frac{1}{2\lambda}x^2, & \text{if } |x| \leq \sqrt{2\lambda},\\
	1, & \text{if } |x| > \sqrt{2\lambda};
	\end{cases}
\]
see for instance \cite[Example 7.3]{luo2024level}.
Since \(\norm{\cdot}_0\) is variationally convex at every point and for every associated subgradient (cf. \cref{eg:zero norm vc}), \cref{thm:vc} implies that its proximal hull inherits the corresponding convex subgradient inequality. Consequently, \cref{thm:vc} provides a conceptual explanation for why the MCP satisfies \eqref{eq:hull vc ineq} despite being nonconvex.

\subsection{An epigraphical characterization of variational convexity}

We now turn to a geometric characterization of variational convexity. As in classical convex analysis, where convexity of a function is equivalent to convexity of its epigraph, variational convexity of a function can be characterized in terms of the variational convexity of its epigraph. To this end, we first introduce the notion of variational convexity for sets.

\begin{definition}
	A nonempty closed set \(C\subseteq\Rn\) is said to be \emph{variationally convex} at \(\bar x\in C\) for \(\bar v\in N_C(\bar x)\), if there exist open convex neighborhoods \(U\ni\bar x\) and \(V\ni\bar v\) such that, for every
	\(
		(x,v)\in\gra N_C\cap\left[U\times V\right]
	\),
	\begin{equation}\label{eq:vc set}
		(\forall x'\in C\cap U)\quad
		\ip{v}{x'-x}\leq0.
	\end{equation}
\end{definition}
\begin{remark}
	Clearly, a nonempty closed set \(C\) is variationally convex at \(\bar x\) for \(\bar v\in N_C(\bar x)\) if and only if the indicator function \(\delta_C\) is variationally convex at \(\bar x\) for \(\bar v\in\lsubdiff\delta_C(\bar x)\); see \cref{thm:variational cvx}.
\end{remark}

We provide below several examples of variationally convex sets.
\begin{example}
	Consider
	\[
	C=:
	\set{(x,\alpha)\in\R^2}[\alpha\geq\sqrt{|x|}].
	\]
	Then \(C\) is variationally convex at \((0,0)\) for \((0,-1)\).
\end{example}
\begin{proof}
	Notice first that for \((x,\alpha)\in C\)
	\begin{equation}\label{eq:normal cone}
		N_C(x,\alpha)
	=
		\begin{cases}
			\R\times\R_-, &x=\alpha=0,\\
			\R_+(1,-2\sgn(x)|x|^{1/2}),& x\neq0, \alpha=\sqrt{|x|},\\
			\set{0},&\text{otherwise}.
		\end{cases}
	\end{equation}
	Take \((x,\alpha)\in C\) and \((v,\beta)\in N_C(x,\alpha)\) satisfying
	\begin{equation}\label{eq:vc constraint}
		\norm{(x,\alpha)}<1/16
		\text{ and }
		\norm{(v,\beta)-(0,-1)}<1/2.
	\end{equation}
	Our goal is to show that
	\begin{align}
		(\forall (x',\alpha')\in C \text{ with }\norm{(x',\alpha')}<1/16) \text{ it holds }\ip{(x'-x,\alpha'-\alpha)}{(v,\beta)}\leq0.\label{eq:vc set eg}
	\end{align}
	When \(\alpha>\sqrt{|x|}\), \eqref{eq:vc set eg} holds trivially and therefore it suffices to consider the case \(\alpha=\sqrt{|x|}\).
	We claim that under \eqref{eq:vc constraint},
	\begin{equation}
		\sqrt{|x|}=\alpha
	\Rightarrow
		x=\alpha=0.
	\end{equation}
	Indeed, suppose that \(x\neq0\). In this case, invoking \eqref{eq:normal cone}, it yields
	\(
		|\beta|
	=
		2|v||x|^{1/2}
	<
		1/2,
	\)
	but \eqref{eq:vc constraint} entails that \(-3/2<\beta<-1/2\), which is absurd.
	Hence \eqref{eq:vc set eg} reduces to
	\begin{equation}\label{eq:vc set eg1}
			(\forall (x',\alpha')\in C\text{ with }\norm{(x',\alpha')}<1/16) \text{ it holds }
			\ip{x'}{v}+\alpha'\beta\leq0.
	\end{equation}
	It is easy to see that, for every $(x',\alpha')\in C$ with $\norm{(x',\alpha')}<1/16$,
	\[
	\ip{x'}{v}\leq|x'||v|\leq|x'|/2\leq|x'|^{1/2}/2\leq\alpha'/2<-\alpha'\beta,
	\]
	where the last inequality holds because \(\beta<-1/2\), furnishing the sought \eqref{eq:vc set eg1}.
 \end{proof}

\begin{example}\label{eg:zero norm level set}
For \(r\geq0\), 	define 
$$C:=\set{x\in\Rn}[\norm{x}_0\leq r].$$
	Let \(\bar x\in C\) be such that \(\norm{\bar x}_0=r\).
	Then \(C\) is variationally convex at \(\bar x\) for every \(\bar v\in N_C(\bar x)\).
\end{example}
\begin{proof}
Define
\[
	(\forall x\in\Rn)\quad
	\supp(x):=\spn\set{e_i}[i\in \mathcal I(x)],
\]
where \(e_i\in\Rn\) denotes the \(i\)-th canonical basis vector. Then \cite[Theorem 3.9]{bauschke2014restricted} states that
\begin{equation}\label{eq:zero cone}
		(\forall x\in C)\quad
		N_C(x)
		=
		\set{v\in\Rn}[\norm{v}_0\leq n-r]
		\cap
		(\supp(x))^\perp.
\end{equation}

	Pick \(\varepsilon>0\) and \((x,v)\in\gra N_C\) such that \(\norm{x-\bar x}<\varepsilon\) and \(\norm{v-\bar v}<\varepsilon\).
	Our goal is to show that
	\begin{equation}
		(\forall x'\in C\cap(\bar x+\varepsilon\ball))\quad
		\ip{v}{x'-x}\leq0.
	\end{equation}
	Shrinking \(\varepsilon>0\) if necessary, we can assume without loss of generality that
	\begin{equation}\label{eq:I}
		(\forall \norm{x-\bar x}<\varepsilon)\quad
		\mathcal I(\bar x)\subseteq \mathcal I(x), \ \mbox{thus}, \
		\norm{x}_0\geq \norm{\bar x}_0=r.
	\end{equation}
	Hence, in particular,
	\begin{equation*}
		(\forall x\in C\cap(\bar x+\varepsilon\ball))\quad
        \norm{x}_0=r = \norm{\bar x}_0.
	\end{equation*}
Taking into account \eqref{eq:I}, we must in fact have
    \begin{equation}\label{eq:samne I}
        (\forall x\in C\cap(\bar x+\varepsilon\ball))\quad
    \mathcal I(x)=\mathcal I(\bar x).
    \end{equation}
Note that \((\supp(\bar x))^\perp=\set{v\in\Rn}[v_i=0\text{ for }i\in\mathcal{I}(\bar x) ]\).
    Then \(\norm{\bar x}_0=r\) implies \(\norm{v}_0\leq n-r\) for every \(v\in(\supp(\bar x))^\perp\), and, invoking \eqref{eq:zero cone}, yields
	\(
	N_C(\bar x)=\left(\supp(\bar x)\right)^\perp
	\).
    For every \( x'\in C\cap(\bar x+\varepsilon\ball)\) and our chosen \((x,v)\in\gra N_C\),
    \eqref{eq:samne I} entails
    \(
        \mathcal I(x)=\mathcal I(x')=\mathcal I(\bar x)
     \),
 so we have
	\[
		v\in N_C(x)=N_C(\bar x)=\left(\supp(\bar x)\right)^\perp
	\text{ and }
		x,x'\in\supp(\bar x),
	\]
    enforcing \(\ip{v}{x'-x}=0\).
\end{proof}

\begin{remark}
The set \(C\) in \cref{eg:zero norm level set} may fail to be variationally convex at \(\bar x\) with \(\norm{\bar x}_0<r\) for \(\bar v\in N_C(\bar x)\). Choose $r:=1$, thus,
	\[
	C=\set{x\in\R^2}[\norm{x}_0\leq 1].
	\]
	Then, according to \eqref{eq:zero cone},
	\[
	N_C(x)
	=
	\begin{cases}
	\set{v\in\R^2}[\norm{v}_0\leq1],&\text{ if }\norm{x}_0=0,\\
    (\supp(x))^\perp,&\text{ if }\norm{x}_0=1.
	\end{cases}
	\]
    Let \(\bar x=(0,0)\), \(\bar v=(0,0)\in N_C(\bar x)\), and let \(
    \varepsilon>0
    \) be arbitrary.
	Choose \(v=(\varepsilon,0)\), \(x=(0,\varepsilon)\), and \(x'=(\varepsilon,0)\).
	Then \(v\in N_C(x)\) and
	\(
	\ip{v}{x'-x}=\varepsilon^2>0
	\).
\end{remark}
Let \(C\subseteq\Rn\) be a nonempty closed set and let \(\bar x\in C\).
We say that \(\bar v\in\Rn\) is a \emph{proximal normal vector} to \(C\) at \(\bar x\), denoted by \(\bar v\in N^P_C(\bar x)\), if there exists \(\tau>0\) such that \(\bar x\in \proj_C(\bar x+\tau \bar v)\); see \cite[Example 6.16]{rockafellar_variational_1998}.
Proximal normal cones will play an important role in the forthcoming epigraphical characterization of variational convexity.

\begin{theorem}\label{thm:epi}
	A proper and lower semicontinuous function \(\func{f}{\Rn}{\RE}\) is variationally convex at \(\bar x \in \dom f\) for \(\bar v\in\lsubdiff f(\bar x)\) if and only if its epigraph \(\epi f\) is variationally convex at \((\bar x,f(\bar x))\) for \((\bar v,-1)\in N_{\epi f}(\bar x,f(\bar x))\).
\end{theorem}
\begin{proof}
	Assume without loss of generality that \(f(\bar x)=0\) and \(\bar x=\bar v=0\).
	\begin{newitemize}
		\item``\(\Rightarrow\)'': The variational convexity of \(f\) at \(\bar x=0\) for \(\bar v=0\) allows us to choose  \(0<\varepsilon<1/2\) such that,
	for every \((x,u)\in\gra\lsubdiff f\) satisfying \(\norm{x}<\varepsilon\), \(f(x)<\varepsilon\), and \(\norm{u}<2\varepsilon\), we have the following
	\begin{equation}\label{eq:epi1}
		(\forall \norm{x'}<\varepsilon)\quad
		f(x')\geq f(x)+\ip{u}{x'-x}.
	\end{equation}
In view of the lower semicontinuity of \(f\) at \(\bar x=0\), there exists \(\varepsilon'>0\) such that \(f(x)>f(\bar x)-\varepsilon=-\varepsilon\) for every \(\norm{x}<\varepsilon'\). 
Pick  \((x,\beta)\in\epi f\), and \((v,b)\in N_{\epi f}(x,\beta)\) satisfying
	\begin{equation}\label{eq:epi}
		% \norm{x}<\varepsilon,\beta<\varepsilon,
		\norm{(x,\beta)}<\min(\varepsilon,\varepsilon') \ \text{and} \
		\norm{(v,b)-(0,-1)}<\varepsilon.
	\end{equation}
Write \(X:=\{(x,\beta) \in \Rn \times \R | \norm{(x,\beta)}<\min(\varepsilon,\varepsilon')\}\) for simplicity.
	We aim to show that
	\begin{equation}\label{eq:set vc}
		(\forall (x',\alpha)\in\epi f\cap X)\quad
		\ip{(v,b)}{(x',\alpha)-(x,\beta)}\leq0.
	\end{equation}
	The choice of \(\varepsilon\) ensures that \(-3/2<b<-1/2\).
	So \((-v/b,-1) \in N_{\epi f}(x,\beta)\).
	We claim that for our chosen \((v,b)\in N_{\epi f}(x,\beta)\)
	\begin{equation}\label{eq:proximal cone claim}
		f(x)=\beta,
	\end{equation}
	and, consequently, by the epigraphical characterization of limiting subgradient (see, e.g., \cite[Theorem 8.9]{rockafellar_variational_1998}),
	\begin{equation}\label{eq:normal is subgrad}
		-v/b\in\lsubdiff f(x)
		\text{ with }
		\norm{v/b}<2\norm{v}<2\varepsilon.
	\end{equation}
Hence, \eqref{eq:epi1} furnishes
	\begin{equation*}
		(\forall (x',\alpha)\in\epi f\cap X)\quad
		\alpha
	\geq
		f(x')\geq f(x)-\ip{v/b}{x'-x}
	=
		\beta-\ip{v/b}{x'-x},
	\end{equation*}
entailing \eqref{eq:set vc}.

	It remains to justify \eqref{eq:proximal cone claim}.\
	Suppose to the contrary that \(\beta>f(x)\).
	Since \(f\) is variationally convex at \(0\) for \(0\), it is prox-regular at \(0\) for \(0\).
	Invoking \cite[Theorem 3.5]{poliquin1996} yields that \(\epi f\) is prox-regular at \((0,0)\) for \((0,-1)\).\
	Then, shrinking \(\varepsilon>0\) if necessary and adjusting \(\varepsilon'>0\), which depends on \(\varepsilon\), we may assume without loss of generality that there exists \(\tau>0\) such that
	\begin{equation}\label{eq:prox-regular set}
		\begin{aligned}
			&(x,\beta)
		\text{ is the unique projection of }
			(x,\beta)+\tau(v,b) \text{ onto }\set{(x',\alpha)\in\epi f}[\norm{(x',\alpha)}<\sqrt{2}\varepsilon];
		\end{aligned}
	\end{equation}
	see, e.g., \cite[Definition 2.10]{poliquin1996}.\
	Recall that \(b<0\) and choose \(0<\tau'<\min((f(x)-\beta)/b,\tau)\).\
	Then \(\alpha:=\beta+\tau' b\) satisfies \(f(x)<\alpha<\beta\), and therefore \(|\alpha|<\max(|\beta|,|f(x)|)\) and \((x,\alpha)\in\epi f\).

Before invoking \eqref{eq:prox-regular set}, we show first that \(|\alpha|<\varepsilon\). Indeed, we already know from \eqref{eq:epi} that \(|\beta|<\varepsilon\) and \(\norm{x}<\min(\varepsilon,\varepsilon')\).
    Consequently \(f(x)\leq\beta<\varepsilon\) because \((x,\beta)\in\epi f\) and \(f(x)>-\varepsilon\) due to \(\norm{x}<\varepsilon'\), which means that \(|f(x)|<\varepsilon\). Altogether, we have shown that \(|\beta|<\varepsilon\) and \(|f(x)|<\varepsilon\), hence \(|\alpha|<\varepsilon\) as claimed. Then it holds that
	\[
		(x,\alpha)\in\set{(x',\alpha)\in\epi f}[\norm{(x',\alpha)}<\sqrt{2}\varepsilon].
	\]
Due to \eqref{eq:prox-regular set}, \((x,\beta)\) is the unique projection of \((x,\beta)+\tau'(v,b)\) onto $\set{(x',\alpha)\in\epi f}[\norm{(x',\alpha)}<\sqrt{2}\varepsilon]$, for every \(\tau'<\tau\), which implies
	\begin{equation*}
		\norm{\tau'(v,b)}
	<
		\norm{(x,\alpha)-\left[(x,\beta)+\tau'(v,b)\right]}
	=
		\norm{(-\tau' v,\alpha-\beta-\tau' b)}.
	\end{equation*}
	This means that
	\(
		|\tau' b|<|\alpha-\beta-\tau' b|=0
	\), which is absurd.
	\item ``\(\Leftarrow\)'':
	This is the easy direction.
	We aim to find \(\hat\varepsilon>0\) such that, for every \((x,v)\in\gra\lsubdiff f\) satisfying \(\norm{x}<\hat\varepsilon\), \(f(x)<\hat\varepsilon\), and \(\norm{v}<\hat\varepsilon\), one has
	\begin{equation}\label{eq:vc set converse:goal}
		(\forall \norm{x'}<\hat\varepsilon)\quad
		f(x')\geq f(x)+\ip{v}{x'-x}.
	\end{equation}
Once this is established, the proof is completed by \cref{thm:variational cvx:4}. 
According to the assumption, there exists \(\varepsilon>0\) such that, for every \((v,b)\in N_{\epi f}(x,\beta)\) satisfying
	\begin{equation}\label{eq:vc set converse}
		\norm{(x,\beta)}<\varepsilon,
		\norm{(v,b)-(0,-1)}<\varepsilon,
	\end{equation}
	it holds that
	\begin{equation}\label{eq:vc set converse1}
		(\forall (x',\alpha)\in\epi f\cap\{(x,\beta) \in \Rn \times \R | \norm{(x,\beta)}< \varepsilon\})\quad
		\ip{(v,b)}{(x',\alpha)-(x,\beta)}\leq0.
	\end{equation}
	Now let \(\hat\varepsilon<\min(\sqrt{\varepsilon}/2,\varepsilon/2)\), and choose \((x,v)\in\gra\lsubdiff f\) satisfying
	\(\norm{x}<\hat\varepsilon\), \(f(x)<\hat\varepsilon\), and \(\norm{v}<\hat\varepsilon\). 
	Then \((v,-1)\in N_{\epi f}(x,f(x))\), by the epigraphical description of limiting subdifferential \cite[Theorem 8.9]{rockafellar_variational_1998}, and \eqref{eq:vc set converse} holds with \(\beta:=f(x)\) and \(b:=-1\).
    
Choose \(x'\) such that \(\norm{x'}<\hat\varepsilon\) and consider two cases.
	Suppose first that \(f(x')<\varepsilon\).
    Then, shrinking \(\hat\varepsilon\) if necessary, the lower semiconinuity of \(f\) at \(\bar x=0\) allows us to assume without loss of generality that \(f(x')>f(\bar x)-\varepsilon=-\varepsilon\), and, consequently, \(|f(x')|<\varepsilon\).
    Invoking \eqref{eq:vc set converse1} with \(\alpha=f(x')\), \(\beta=f(x)\), and \(b=-1\), entails
	\begin{equation*}
		\ip{(v,-1)}{(x',f(x'))-(x,f(x))}\leq0,
	\end{equation*}
	which amounts to \eqref{eq:vc set converse:goal}.
	Consider now the case \(f(x')\geq\varepsilon\).
	Then
	\begin{equation*}
		f(x)+\ip{v}{x'-x}\leq\hat\varepsilon+\norm{v}(\norm{x'}+\norm{x})<\varepsilon\leq f(x'),
	\end{equation*}
    where the second inequality holds, since
    \(
    \norm{v}(\norm{x'}+\norm{x})<2\hat\varepsilon^2
    \).\qedhere
	\end{newitemize}
\end{proof}

\section{Calculus rules for variational convexity}\label{sec:calculus}
\subsection{Composition rules}
Having developed new characterizations of variational convexity, we now turn to its calculus rules, which facilitate the verification of variational convexity in practice. Several elementary calculus rules have been established in \cite{luo2024level}, including a separable sum rule \cite[Proposition 5.14]{luo2024level} and a sum rule involving a variationally convex function and a differentiable convex function \cite[Lemma 5.5]{luo2024level}.

The following lemma was used in \cite[Theorem 3.1]{poliquin2010calculus} to establish a composition rule for prox-regularity.
We use it here to derive a composition rule for variational convexity.
A proof is provided in the appendix for the sake of completeness.
% \todo{Comment on CQs}
\begin{lemma}\label{lem:chain rule cq}
	Let \(\func{g}{\R^m}{\RE}\) be proper and lower semicontinuous, let \(\func{F}{\Rn}{\R^m}\) be a \(C^1\) mapping, and let \(f(x)=(g\circ F)(x)\).
	Suppose that for \(\bar x\in F^{-1}(\dom g)\) it holds
	\begin{equation}\label{eq:chain rule cq}
		\ker\nabla F(\bar x)^*
	\cap
		\partial^\infty g(F(\bar x))
	=0.
	\end{equation}
	Then the following hold:
	\begin{enumerateq}
		\item\label{lem:chain rule cq::1}
		There exists \(\varepsilon>0\) such that for every \(x \in \Rn \) satisfying \(\norm{x-\bar x}<\varepsilon\) and \(f(x)<f(\bar x)+\varepsilon\)
		it holds that
		\(
			\ker\nabla F(x)^*
		\cap
			\partial^\infty g(F(x))
		=0
		\), and, consequently,
		\[
			\partial f(x)\subseteq \nabla F(x)^*\partial g(F(x)).
		\]
		\item\label{lem:chain rule cq::2}
		The mapping \(\ffunc{S}{\Rn\times \Rn}{\R^m}\) defined by
		\[
			S(x,v)
		=
			\set{u\in\partial g(F(x))}[\nabla F(x)^*u=v]
		\]
		is outer semicontinuous and locally bounded at \((\bar x,\bar v)\) for every \(\bar v\in\partial f(\bar x)\) with respect to \(f\)-attentive convergence.
		Consequently, \(S\) is upper semicontinuous  at \((\bar x,\bar v)\) for every \(\bar v\in\partial f(\bar x)\) with respect to \(f\)-attentive convergence at, that is, for every open set \(K\supseteq S(\bar x,\bar v)\) there exists \(\varepsilon>0\) such that
		for every \((x,v) \in \Rn \times \Rn\) satisfying \(\norm{x-\bar x}<\varepsilon\), \(f(x)<f(\bar x)+\varepsilon\), and  \(\norm{v-\bar v}<\varepsilon\), it holds that
		\[
			S(x,v)\subseteq K.
		\]
	\end{enumerateq}
\end{lemma}
We now present a nonlinear composition rule for variational convexity. As the first main result of this section, it is sufficiently general to encompass many useful special cases that will be examined later. To put this result into perspective, recall that prox-regularity is preserved under nonlinear composition provided that the inner mapping is \(C^2\); see \cite[Theorem 3.1]{poliquin2010calculus}. In contrast, for variational convexity, the inner mapping \(F\) need not to be \(C^2\). This greater flexibility is achieved at the expense of imposing a local convexity assumption.
    
\begin{theorem}[nonlinear composition]\label{thm:nonlinear comp}
	Let \(\func{g}{\R^m}{\RE}\) be proper and lower semicontinuous, let \(\func{F}{\Rn}{\R^m}\) be a \(C^1\) mapping, and let \(f(x)=(g\circ F)(x)\). Let \(\bar x\in F^{-1}(\dom g)\) and \(\bar v\in\lsubdiff f(\bar x)\) be such that \(g\) is variationally convex at \(F(\bar x)\) for every \(\bar u\in\lsubdiff g(F(\bar x))\) satisfying
    \(
    \bar v=\nabla F(\bar x)^*\bar u.
    \)
	If the constraint qualification
	\begin{equation}\label{eq:chain cq}
		\ker \nabla F(\bar x)^*\cap\partial^\infty g(F(\bar x))=\set{0}
	\end{equation}
	is satisfied and there exist open convex neighborhoods \(U\ni\bar x\) and \(V\ni\bar v\),  and \(\varepsilon>0\) such that, for every $(x,v) \in \gra\lsubdiff f\cap[U_\varepsilon \times V]$ and every $u \in S(x,v)$,
	\begin{equation}\label{eq:Gu convex}
		\ip{u}{F(\cdot)}\text{ is convex on }U,
	\end{equation}
where
    \[
		\ffunc{S}{\Rn\times\Rn}{\R^m}:(x,v)\mapsto\set{u\in\lsubdiff g(F(x))}[v=\nabla F(x)^*u],
	\]
	then \(f\) is variationally convex at \(\bar x\) for \(\bar v\in\lsubdiff f(\bar x)\).
\end{theorem}
\begin{proof}
	\cref{lem:chain rule cq} entails that there exists \(\varepsilon_1>0\) such that for every \(x\in\Rn\) satisfying
	\(\norm{x-\bar x}<\varepsilon_1\) and \(f(x)<f(\bar x)+\varepsilon_1\),
	\begin{equation}\label{eq:cq}
		\ker\nabla F(x)^*\cap\partial^\infty g(F(x))=\set{0},
        \text{ and thus }
        \lsubdiff f(x)\subseteq\nabla F(x)^*\lsubdiff g(F(x)),
	\end{equation}
	and that the mapping \(S\)	is upper semicontinuous at \((\bar x,\bar v)\) with respect to \(f\)-attentive convergence.
    Moreover the set \(S(\bar x,\bar v)\) is compact by the closedness of \(\lsubdiff g(F(\bar x))\) and the local boundedness of \(S\) at \((\bar x,\bar v)\) with respect to \(f\)-attentive convergence (recall \cref{lem:chain rule cq::2}).
    Note that \(g\) is variationally convex at \(F(\bar x)\) for every \(\bar u\in \lsubdiff g(F(\bar x))\) satisfying \(\bar v=\nabla F(\bar x)^*\bar u\).
    Therefore to each \(\bar u\in S(\bar x,\bar v)\) corresponds some \(\varepsilon_{\bar u}>0\) such that for \((y,u)\in\gra\lsubdiff g\) satisfying
    \begin{equation*}
        \norm{y-F(\bar x)}<\varepsilon_{\bar u}, g(y)<g(F(\bar x))+\varepsilon_{\bar u}, \text{ and } \norm{u-\bar u}<\varepsilon_{\bar u},
    \end{equation*}
    one has
    \begin{equation*}
        (\forall \norm{y'-F(\bar x)}<\varepsilon_{\bar u})\quad
        g(y')\geq g(y)+\ip{u}{y'-y}.
    \end{equation*}
    Owing to the compactness of \(S(\bar x,\bar v)\), we may take a finite cover \(\cup_{i\in I}(\bar u_i+(\varepsilon_{\bar u_i}/2)\ball)\supseteq S(\bar x,\bar v)\) for finitely many \(\bar u_i\in S(\bar x,\bar v)\) and a finite index set \(I\).
    Therefore, in particular, the inequality described above holds with the radius \(\varepsilon_{\bar u}\) being \(\varepsilon_{\bar u_i}\) for each \(i\in I\).
    Define \(\varepsilon_2:=\min_{i\in I}\varepsilon_{\bar u_i}/2>0\).
    Then for every \((y,u)\in\gra\lsubdiff g\) satisfying
	\begin{equation}\label{eq:g vc condition}
		\norm{y-F(\bar x)}<\varepsilon_2, g(y)<g(F(\bar x))+\varepsilon_2, \text{ and } \dist(u,S(\bar x,\bar v))<\varepsilon_2,
	\end{equation}
	it holds that \(\norm{y-F(\bar x)}<\varepsilon_{\bar u_i}\), \(g(y)<g(F(\bar x))+\bar\varepsilon_i\), and \(\norm{u-\bar u_i}<\varepsilon_{\bar u_i}\) for some \(i\in I\), and consequently
	\begin{equation}\label{eq:g vc}
		(\forall \norm{y'-F(\bar x)}<\varepsilon_2)\quad
		g(y')\geq g(y)+\ip{u}{y'-y}.
	\end{equation}
	Invoking the upper semicontinuity of \(S\) at \((\bar x,\bar v)\) yields that there exists \(\varepsilon_3>0\) such that, for every
    \(\norm{x-\bar x}<\varepsilon_3, f(x)<f(\bar x)+\varepsilon_3\) and \(\norm{v-\bar v}<\varepsilon_3\), it holds that
	\begin{equation}\label{eq:usc again}
	% \begin{rcases}
	% 	\norm{x-\bar x}<\varepsilon_4,\norm{v-\bar v}<\varepsilon_4&~\\
	% 	f(x)<f(\bar x)+\varepsilon_4&~\\
	% \end{rcases}
	% \Rightarrow
		S(x,v)\subseteq S(\bar x,\bar v)+\varepsilon_2\ball.
	\end{equation}
    Choose \(\tilde \varepsilon<\min(\varepsilon_1,\varepsilon_2,\varepsilon_3,\varepsilon)\) and define
	\[
	\widetilde U:=(\bar x+ \tilde \varepsilon\ball)\cap U, \widetilde U_{\tilde \varepsilon}:=\set{x\in U}[f(x)<f(\bar x)+ \tilde \varepsilon], \text{ and } \widetilde V:=(\bar v+  \tilde \varepsilon\ball)\cap V.
	\]
	Shrinking \(\tilde \varepsilon>0\) if necessary, the continuity of \(F\) allows us to assume without loss of generality that
	\begin{equation}\label{eq: F close}
		(\forall x\in \widetilde U)\quad
		\norm{F(x)-F(\bar x)}<\varepsilon_2.
	\end{equation}
    Pick \((x,v)\in\gra\lsubdiff f\cap[\widetilde U_{\tilde \varepsilon} \times \widetilde V]\).
	We aim to show that
	\begin{equation}\label{eq:vc of comp}
		(\forall x'\in \widetilde U)\quad
		f(x')\geq f(x)+\ip{v}{x'-x},
	\end{equation}
	which will furnish the variational convexity of \(f\) at \(\bar x\) for \(\bar v\) through \cref{thm:variational cvx:4}.
	To this end, we will invoke \eqref{eq:g vc} after first validating its requirements.
    Notice first that \eqref{eq:cq} means that there exists \(u\in\lsubdiff g(F(x))\) such that \(v=\nabla F(x)^*u\), while \eqref{eq:usc again} guarantees that
	\begin{equation}\label{eq:grad close}
		\dist(u,S(\bar x,\bar v))<\varepsilon_2.
	\end{equation}
	Moreover, \(g(F(x))=f(x)<f(\bar x)+\varepsilon_2=g(F(\bar x))+\varepsilon_2\), which combined with \eqref{eq: F close} and \eqref{eq:grad close} means that the pair \((F(x),u)\in\gra\lsubdiff g\) satisfies \eqref{eq:g vc condition}.
    Moreover \(\norm{F(x')-F(\bar x)}<\varepsilon_2\) for every \(x'\in \widetilde U\) by appealing to \eqref{eq: F close} again.
    Altogether, \eqref{eq:g vc} is applicable with \(y'=F(x')\) and \(y=F(x)\), furnishing
	\begin{align*}\label{eq:composition1}
    (\forall x'\in \widetilde U)\quad
		f(x')-f(x)
    &=
        g(F(x'))-g(F(x))\geq\ip{u}{F(x')-F(x)}\\
    &\geq
        \ip{\nabla F(x)^*u}{x'-x}
        =
        \ip{v}{x'-x},
	\end{align*}
    where the last inequality follows from the fact that \((x,v)\in \gra\lsubdiff f\cap[\widetilde U_\varepsilon\times \widetilde V]\) and from the convexity of the differentiable mapping \(x\mapsto\ip{u}{F(x)}\) on \(\widetilde U\subseteq U\).
\end{proof}

The constraint qualification \eqref{eq:chain cq} is standard; see, for example, the constraint qualification appearing in the subdifferential chain rule \cite[Theorem 10.6]{rockafellar_variational_1998}.

As a consequence of \cref{thm:nonlinear comp}, variationally convex functions can be characterized as compositions of convex functions and $C^1$ mappings satisfying a conic convexity condition; see \cref{cor:conic} below. This characterization is a natural analogue of the theory of strongly amenable functions, which are automatically prox-regular; see \cite[Proposition 2.5]{poliquin1996}.
Recall that the polar cone of a set \(K\subseteq\Rn\) is
\[
K^\ominus
\defeq
\set{u\in\Rn}[(\forall x\in K)~\ip{u}{x}\leq0],
\]
see \cite[Definition 6.22]{BC}.

\begin{corollary}\label{cor:conic}
	Let \(\func{g}{\R^m}{\R}\) be convex, let \(\func{F}{\Rn}{\R^m}\) be a \(C^1\) mapping, and let \(f(x)=(g\circ F)(x)\). Let \(\bar x\in F^{-1}(\dom g)\) and \(\bar v\in\lsubdiff f(\bar x)\). If there exists an open set \(K\supseteq S(\bar x,\bar v)\) such that \(F\) is \(K^{\ominus}\)-convex locally around \(\bar x\), that is, there exists an open convex neighborhood \(U\ni\bar x\) such that
\[
(\forall x_0,x_1\in U)(\forall 0<\lambda<1)\quad
    F((1-\lambda)x_0+\lambda x_1)
\in
    (1-\lambda)F(x_0)+\lambda F(x_1)+K^\ominus,
\]
    where
    \[
		\ffunc{S}{\Rn\times\Rn}{\R^m}:(x,v)\mapsto\set{u\in\lsubdiff g(F(x))}[v=\nabla F(x)^*u],
	\]
    then \(f\) is variationally convex at \(\bar x\) for \(\bar v\in\lsubdiff f(\bar x)\).
\end{corollary}
\begin{proof}
	Convexity of \(\func{g}{\R^m}{\R}\) ensures that \(g\) is variationally convex at \(F(\bar x)\) for every \(\bar u\in\lsubdiff g(F(\bar x))\), and \(\partial^\infty g(y)=\set{0}\) for every $y \in \R^m$. Consequently, the constraint qualification \eqref{eq:chain cq} is automatically satisfied; see, e.g., \cite[Theorem 9.13]{rockafellar_variational_1998}.
    According to \cref{lem:chain rule cq::2}, there exists \(\varepsilon>0\) such that, for every \((x,v)\in\gra\lsubdiff f\) satisfying
	\(
		\norm{x-\bar x}<\varepsilon, f(x)<f(\bar x)+\varepsilon, \text{ and } \norm{v-\bar v}<\varepsilon,
	\)
	it holds that
	\begin{equation}\label{eq:usc}
		S(x,v)\subseteq K.
	\end{equation}
	We now justify that \eqref{eq:Gu convex} holds. Let \(U\defeq\bar x+\varepsilon\ball\), \(U_\varepsilon\defeq\set{x\in U}[f(x)<f(\bar x)+\varepsilon]\), and \(V\defeq\bar v+\varepsilon\ball\).  Shrinking \(\varepsilon>0\) if necessary, we may assume without loss of generality that \(F\) is \(K^\ominus\)-convex on \(U\).
    Let \((x,v)\in \gra\lsubdiff f\cap[U_\varepsilon\times V]\) and $u \in S(x,v)$. Then $u \in K$.

The \(K^\ominus\)-convexity assumption on \(F\) implies that
	\begin{align*}
		(\forall 0<\lambda<1)(\forall x_0,x_1\in U) \quad F((1-\lambda)x_0+\lambda x_1)-(1-\lambda)F(x_0)-\lambda F(x_1)
	\in
		K^\ominus,
	\end{align*}
hence,
\begin{align*}
		 (\forall 0<\lambda<1)(\forall x_0,x_1\in U) \quad \langle F((1-\lambda)x_0+\lambda x_1), u \rangle \leq (1-\lambda) \langle F(x_0), u \rangle +\lambda \langle F(x_1), u \rangle,
\end{align*}
which proves the convexity of $x \mapsto \ip{F(x)}{u}$ on $U$.    
\end{proof}

When $F$ is a linear mapping, assumption \eqref{eq:Gu convex} in \cref{thm:nonlinear comp} is automatically satisfied, yielding a linear composition rule for variational convexity.

\begin{theorem}[linear composition]\label{thm:chain rule with cq}
	Let \(\func{g}{\R^m}{\RE}\) be proper and lower semicontinuous, let \(A\in\R^{m\times n}\), and let \(f(x)=g(Ax)\).
	Let \(\bar x\in A^{-1}(\dom g)\) and let \(\bar v\in\lsubdiff f(\bar x)\) be such that \(g\) is variationally convex at \(A\bar x\) for every \(\bar u\in\partial g(A\bar x)\) satisfying \(\bar v=A^T\bar u\).
	If the constraint qualification
	\begin{equation}\label{eq:chain rule cq2}
		\ker A^T\cap\partial^\infty g(A\bar x)=\set{0}
	\end{equation}
is satisfied, then \(f\) is variationally convex at \(\bar x\) for \(\bar v\in\partial f(\bar x)\).
\end{theorem}

\begin{corollary}\label{thm:vc composition}
	Let \(\func{g}{\Rn}{\RE}\) be proper and lower semicontinuous, let \(A\in\R^{m\times n}\) be surjective, and let \(f(x)=g(Ax)\). Let \(\bar x\in A^{-1}(\dom g)\). If \(g\) is variationally convex at \(A\bar x\) for \(\bar u\in\partial g(A\bar x)\), then \(f\) is variationally convex at \(\bar x\) for \(A^T\bar u\in\partial f(\bar x)\).
\end{corollary}

When \(g\) is locally Lipschitz around \(A\bar x\), it holds that \(\partial^\infty g(A\bar x)=\set{0}\) (see \cite[Theorem 9.13]{rockafellar_variational_1998}) yielding the following corollary.

\begin{corollary}
    Let \(\func{g}{\R^m}{\RE}\) be proper and lower semicontinuous, let \(A\in\R^{m\times n}\), and let \(f(x)=g(Ax)\). Let \(\bar x\in A^{-1}(\dom g)\) and let \(\bar v\in\lsubdiff f(\bar x)\) be such that \(g\) is variationally convex at \(A\bar x\) for every \(\bar u\in\partial g(A\bar x)\) with \(\bar v=A^T\bar u\). If \(g\) is locally Lipschitz around \(A\bar x\), then \(f\) is variationally convex at \(\bar x\) for \(\bar v\in\partial f(\bar x)\).
\end{corollary}

We now provide new examples of variationally convex functions by means of \cref{thm:nonlinear comp,thm:vc composition}.

\begin{example}\label{eg:zero norm linear composition}
	Let \(A\in\R^{m\times n}=[a_1,\ldots,a_m]^T\) for \(a_i\in\Rn\) and let \(\bar x\in\Rn\).
	Suppose that
	\begin{equation}\label{eq:zero norm cq}
		\text{the vectors } \set{a_i}_{i\not\in \mathcal I(A\bar x)}
		\text{ are linearly independent}.
	\end{equation}
	Then the function \(f(x)=\norm{Ax}_0\) is variationally convex at \(\bar x\) for every \(\bar v\in\lsubdiff f(\bar x)\).
\end{example}
\begin{proof}
    We begin by establishing that
    \begin{equation}\label{eq:zero norm singular}
    (\forall x\in\Rn)\quad
        \partial^\infty\norm{\cdot}_0(x)
	=
		\set{v\in\Rn}[(\forall i\in \mathcal I(x))~v_i=0],
    \end{equation}
which will be useful in verifying the constraint qualification \eqref{eq:chain rule cq2}. 
	In view of \eqref{eq:zero norm subdiff},
	\begin{equation}\label{eq:singular zero norm1}
		(\forall \lambda > 0) (\forall x \in \R^n)\quad
		\lambda\lsubdiff\norm{\cdot}_0(x)=\lsubdiff\norm{\cdot}_0(x) = \set{v\in\Rn}[(\forall i\in \mathcal I(x))~v_i=0].
	\end{equation}
    Take \(x\in\Rn\). First, we justify ``\(\subseteq\)'' in \eqref{eq:zero norm singular}. For every \(y\) sufficiently close \(x\), it holds that \( \mathcal I(x)\subseteq \mathcal I(y)\); see, e.g., \cite[Lemma 3.1(ii)\&(viii)]{bauschke2014restricted}. Hence, for every \(y\) sufficiently close \(x\),
	\begin{equation}\label{eq:singular zero norm2}
		\lsubdiff\norm{\cdot}_0(y)
	\subseteq
		\lsubdiff \norm{\cdot}_0(x).
	\end{equation}
	Pick \(v\in\partial^\infty\norm{\cdot}_0(x)\) and consider the sequences $\lambda_k\downarrow0$ and $(y_k,v_k)\in\gra\lsubdiff \norm{\cdot}_0$ such that
	\[
		y_k\to x\text{ and }\lambda_kv_k\to v
        \text{ with }\norm{y_k}_0\to \norm{x}_0
		\text{ as }k\to+\infty.
	\]
	Then, for \(k\) sufficiently large, \eqref{eq:singular zero norm1} and \eqref{eq:singular zero norm2} imply
	\[\lambda_kv_k\in\lsubdiff\norm{\cdot}_0(y_k)\subseteq\lsubdiff\norm{\cdot}_0(x)=\set{v\in\Rn}[(\forall i\in \mathcal I(x))~v_i=0],
	\]
	furnishing the requested inclusion. To see the converse inclusion ``\(\supseteq\)'', choose \(v\in\Rn\) such that \(v_i=0\), for every $i\in \mathcal I(x)$, and define
    \[
    (\forall k\in\N)\quad
        y_k
    \defeq
        x+\tfrac{1}{k}\sum_{i\in \mathcal{I}(x)}x_ie_i
    \text{ and }
        v_k
    \defeq
        kv+\sum_{i\notin \mathcal{I}(x)}e_i,
    \]
    where \(e_i\in\Rn\) denotes the \(i\)-th canonical basis.
    Then it is easy to see that \(\mathcal{I}(x)=\mathcal{I}(y_k)\), \(\norm{y_k}_0=\norm{x}_0\), and \(v_k\in\lsubdiff\norm{\cdot}_0(y_k)\).
    Since \(y_k\to x\) and \(v_k/k\to v\) as \(k\to\infty\), it follows that \(v\in\partial^\infty\norm{\cdot}_0(x)\).
    
	Turning now to proof of the variational convexity of \(f\), we note first that \(\norm{\cdot}_0\) is variationally convex at \(A\bar x\) for every \(v\in\lsubdiff\norm{\cdot}_0(A\bar x)\); see \cref{eg:zero norm vc}.
	Thus it suffices to show that the constraint qualification \eqref{eq:chain rule cq2} holds.\
	Take \(v=(v_1,\ldots,v_m)\in\partial^\infty\norm{\cdot}_0(A\bar x)\) such that $A^Tv=0$. Then \eqref{eq:zero norm singular} implies \(v_i=0\) for \(i\in \mathcal I(A\bar x)\).In turn, according to \eqref{eq:zero norm cq},
	\[
	A^Tv=\sum_{i\not\in \mathcal I(A\bar x)}a_iv_i=0
	\Rightarrow
	v=0,
	\]
	completing the proof.
\end{proof}
Next, we investigate the variational convexity of the jump function
\begin{equation*}
	x
	\mapsto
	\card\set{1\leq i\leq n-1}[x_i\neq x_{i+1}],
\end{equation*}
which counts the number of jumps in \(x\in\Rn\).
% Given its composition structure, the jump function inherits variational convexity from the zero ``norm''.
\begin{example}\label{eg:jump vc}
	The function
	\begin{equation*}
		\func{f}{\Rn}{\R}:
		x
	\mapsto
		\card\set{1\leq i\leq n-1}[x_i\neq x_{i+1}]
	\end{equation*}
	is variationally convex at every \(\bar x\in\Rn\) for every \(\bar v\in\lsubdiff f(\bar x)\). For every \(x\in\Rn\), it holds
	\begin{equation*}
		\partial^\infty f( x)
	=
		\rsubdiff f( x)
	=
		\lsubdiff f(x)
	=
		\set{D^*v}[v_i=0\text{ if }x_i\neq x_{i+1}],
	\end{equation*}
    where
    \[
	D=
	\begin{bmatrix}
		1 & -1 & 0 & \cdots & 0\\
		0 & 1 & -1 & \cdots & 0\\
		\vdots & \vdots & \vdots & \ddots & \vdots\\
		0 & 0 & 0 & \cdots & -1
	\end{bmatrix} \in\R^{(n-1)\times n}
	\]
is the matrix representation of the difference operator.
\end{example}
\begin{proof}
	Clearly \(D\) is surjective and \(f(x)=\norm{Dx}_0\).\
	Therefore \cref{thm:vc composition} and \cref{eg:zero norm vc} imply that \(f\) is variationally convex at every \(\bar x\in\Rn\) for every \(\bar v\in\lsubdiff f(\bar x)\).
	Moreover, \eqref{eq:zero norm subdiff} and \eqref{eq:zero norm singular} imply that for every \(x\in\Rn\)
	\begin{align*}
	\partial^\infty\norm{\cdot}_0(Dx)
	=
		\rsubdiff\norm{\cdot}_0(Dx)
	=
		\lsubdiff\norm{\cdot}_0(Dx)
	=
		\set{v\in\R^{n-1}}[v_i=0\text{ if }x_i\neq x_{i+1}],
	\end{align*}
	Hence the subdifferential chain rule \cite[Exercise 10.7]{rockafellar_variational_1998}, combined with surjectivity of \(D\), yields the claimed formula of subdifferentials. 
\end{proof}

\subsection{Sum rules}
In this section, we investigate sum rules for variational convexity. These rules are useful for establishing variational convexity in optimization problems involving sums of functions.
\begin{theorem}\label{thm:sum rule}
	Let \(\func{f_1,f_2}{\Rn}{\RE}\) be proper and lower semicontinuous, and let  \(f=f_1+f_2\). Let \(\bar x\in\dom f_1\cap\dom f_2\) and \(\bar v\in\lsubdiff f(\bar x)\) be such that, for every \(\bar v_i\in\partial f_i(\bar x)\) with the property that \(\bar v=\bar v_1+\bar v_2\), the function \(f_i\) is variationally convex at \(\bar x\) for \(\bar v_i\) for \(i=1,2\). If the constraint qualification
	\begin{equation}\label{eq:sum cq}
	\partial^\infty f_1(\bar x)\cap\left(-\partial^\infty f_2(\bar x)\right)=\set{0}
	\end{equation}
is satisfied, then \(f\) is variationally convex at \(\bar x\) for \(\bar v\).
\end{theorem}
\begin{proof}
	Define \(g(y_1,y_2)=f_1(y_1)+f_2(y_2)\) for \((y_1,y_2)\in\R^{n}\times \Rn\), and
	\(A:=
		\begin{bmatrix}
		\Id~\Id
		\end{bmatrix}^T \in \R^{2n\times n}.
	\)
Then the constraint qualification \eqref{eq:sum cq} is nothing else than \eqref{eq:chain rule cq2} with \(g\) and \(A\) defined as above. The statement is a direct consequence of \cref{thm:chain rule with cq}.
\end{proof}

% Recall that a function \(\func{f}{\Rn}{\RE}\) is strictly differentiable at \(\bar x\), if
% \(
% (\forall x,x'\in\Rn)~
% f(x')=f(x)+\ip{v}{x'-x}+o(\norm{x'-x})
% \)
% for some \(v\in\Rn\); see \cite[Definition 9.17]{rockafellar_variational_1998}.
% \begin{corollary}
% 	Let \(\func{f_i}{\Rn}{\RE}\) be proper and lsc for \(i=1,2\), and let \(f_1\) be strictly differentiable at \(\bar x\in\dom f_1\cap\dom f_2\).\
% 	Suppose that \(f_i\) is variationally convex at \(\bar x\) for \(\bar v_i\in\partial f_i(\bar x)\) for each \(i=1,2\).
% 	Then \(f=f_1+f_2\) is variationally convex at \(\bar x\) for \(\bar v=\bar v_1+\bar v_2\in\partial f(\bar x)\).
% \end{corollary}
% The constraint qualification \eqref{eq:sum cq} holds automatically when \(f_1\) is locally Lipschitz around \(\bar x\), or more specifically when \(f_1\) is convex and differentiable, in which case the sum rule \cite[Lemma 5.5]{luo2024level} can be recovered.

If \(f\) is locally Lipschitz around \(\bar x\), then \(\partial^\infty f(\bar x)=\set{0}\) (see \cite[Theorem 9.13]{rockafellar_variational_1998}), in which case the constraint qualification \eqref{eq:sum cq} holds automatically.

\begin{corollary}\label{cor:vc sum}
	Let \(\func{f_1}{\Rn}{\RE}\) be proper and lower semicontinuous, let \(\bar x\in\dom f_1\), and let \(\func{f_2}{\Rn}{\R}\) be locally Lipschitz around \(\bar x\). Let \(f=f_1+f_2\) and  \(\bar v\in\lsubdiff f(\bar x)\) be such that, for every \(\bar v_i\in\partial f_i(\bar x)\) with the property that \(\bar v=\bar v_1+\bar v_2\), the function \(f_i\) is variationally convex at \(\bar x\) for \(\bar v_i\) for \(i=1,2\). Then \(f\) is variationally convex at \(\bar x\) for \(\bar v\in\lsubdiff f(\bar x)\).
\end{corollary}

\begin{example}\label{eg:sum}
We now illustrate how the sum rule in \cref{thm:sum rule} can be used to detect variational convexity in the constrained optimization problem
	\begin{equation}\label{eq:zero ball problem}
		\min_{x \in \Rn} f(x) \quad \text{subject to} \quad \|x\|_0 \leq r,
	\end{equation}
where \(\func{f}{\Rn}{\RE}\) is a proper, convex and lower semicontinuous function, and  \(r\geq0\).

Let \(\bar x \in \dom f\) be such that \(\norm{\bar x}_0=r\) and satisfy
	\begin{equation}\label{eq:zero norm ball cq}
		N_{\dom f}(\bar x)\cap\left(-\supp(\bar x)\right)^\perp=\set{0}.
	\end{equation}
Then \(\varphi=f+\delta_C\), where $C:=\{x \in \R^n | \norm{\bar x}_0 \leq r\}$, is variationally convex at \(\bar x\) for every \(\bar v\in\lsubdiff\varphi(\bar x)\). In particular, the constraint qualification \eqref{eq:zero norm ball cq} holds automatically when \(\dom f=\Rn\). Indeed, in view of \cite[Proposition 8.12]{rockafellar_variational_1998}, the convexity of \(f\) entails \(\partial^\infty f=N_{\dom f}\).
Recall that \(\lsubdiff \delta_C(\bar x)=\partial^\infty \delta_C(\bar x)=N_C(\bar x)=\supp(\bar x)^\perp\); see \cite[Exercise 8.14]{rockafellar_variational_1998} and \cite[Theorem 3.9]{bauschke2014restricted}.\
	Therefore condition \eqref{eq:zero norm ball cq} entails that \eqref{eq:sum cq} holds.
	Moreover \cref{eg:zero norm level set} shows that \(\delta_C\) is variationally convex at \(\bar x\) for every subgradient thereof, and then so is the convex function \(f\).\
	Invoking \cref{thm:sum rule} completes the proof. 

If  \(0\in\lsubdiff\varphi(\bar x)\), then \(\bar x\) is a local minimizer of the optimization problem \eqref{eq:zero ball problem}.
\end{example}

\subsection{Proximal average}

The preceding calculus rules are primarily based on the nonlinear composition rule in \cref{thm:nonlinear comp}. We now investigate a different construction by showing that the proximal average of variationally convex functions is again variationally convex.

\begin{definition}
	Let \(\func{f,g}{\Rn}{\RE}\) be proper, lower semicontinuous, and prox-bounded functions with thresholds \(\lambda_f,\lambda_g>0\) respectively.
	The \textit{proximal average of \(f,g\) with parameter \(0<\lambda<\min(\lambda_f,\lambda_g)\) and \(0\leq\alpha\leq1\)} is defined as
	\[
		\varphi_\lambda^\alpha(f,g)
	=
		-\env[\lambda]\left(-\alpha\env f-(1-\alpha)\env g\right).
	\]
	We write \(\varphi_\lambda^\alpha\) whenever there is no ambiguity.
\end{definition}
Proximal averages were originally defined for convex functions in \cite{bauschke2008proximal}. As this work concerns nonconvex objects, we employ the more general definition of proximal averages for possibly nonconvex prox-bounded functions introduced in \cite{chen2020proximal}. Proximal averages have found applications in various areas including in machine learning and image processing; see \cite{yu2013better,yu2015minimizing,yao2022low,velasco2022learnable} and the references therein. In particular, proximal averages approximate the Moreau envelopes of nonseparable sums and have been used to solve optimization problems involving such structures; see \cite{yu2015minimizing}. These applications rely on the following properties of proper, lower semicontinuous, and prox-bounded functions \(\func{f,g}{\Rn}{\RE}\)
\begin{equation}\label{eq:prox average}
(\forall 0\leq\alpha\leq 1)\quad
    \env\varphi_\lambda^\alpha
=
	\alpha\env f+(1-\alpha)\env g
\text{ and }
    \prox \varphi_\lambda^\alpha
		=
		\alpha\prox f+(1-\alpha)\prox g,
\end{equation}
see \cite[Theorem 5.4]{chen2020proximal}. We now establish that proximal averages preserve variational convexity, which allows us to generate parameterized families of variationally convex functions.

\begin{theorem}[proximal average]\label{thm: pa vc}
	Let \(\func{f,g}{\Rn}{\RE}\) be proper, lower semicontinuous, and prox-bounded functions. If \(f,g\) are variationally convex at \(\bar x \in \dom f \cap \dom g\) for \(\bar v\in\lsubdiff f(\bar x)\cap\lsubdiff g(\bar x)\), then, for \(\lambda>0\) sufficiently small and every \(0\leq\alpha\leq1\), the proximal average \(\varphi_\lambda^\alpha\) is variationally convex at \(\bar x\) for \(\bar v\).
\end{theorem}
\begin{proof}
We first show that \(\bar v\in\lsubdiff\varphi_\lambda^\alpha(\bar x)\) for every sufficiently small \(\lambda>0\). Since \(f,g\) are variationally convex at \(\bar x\) for \(\bar v\), they are prox-regular \(\bar x\) for \(\bar v\). According to \cite[Proposition 13.37]{rockafellar_variational_1998}, there exists \(\bar\lambda_1>0\) such that, for every \(0<\lambda<\bar\lambda_1\), it holds
	\(
	\bar x=\prox f(\bar x+\lambda\bar v)=\prox g(\bar x+\lambda\bar v)
	\),
	hence
	\begin{align*}
		(\forall 0\leq\alpha\leq 1)\quad
		\bar x
	=
		\alpha\prox f(\bar x+\lambda\bar v)+(1-\alpha)\prox g(\bar x+\lambda\bar v)
	=
		\prox \varphi_\lambda^\alpha(\bar x+\lambda\bar v).
	\end{align*}
By \cite[Theorem 5.1(g)]{chen2020proximal}, \(\varphi_\lambda^\alpha\) is \(\lambda\)-hypoconvex, which means that
	\[
	\bar x=\prox \varphi_\lambda^\alpha(\bar x+\lambda\bar v)=\left(\Id+\lambda\lsubdiff\varphi_\lambda^\alpha\right)^{-1}(\bar x+\lambda\bar v)
	\Rightarrow
	\bar v\in\lsubdiff\varphi_\lambda^\alpha(\bar x).
	\]
The variational convexity of \(f,g\) at \(\bar x\) for \(\bar v\in\lsubdiff f(\bar x)\cap\lsubdiff g(\bar x)\) implies, by \cref{thm:variational cvx:3}, the existence of \(\bar \lambda_2>0\) such that, for every \(0<\lambda<\bar\lambda_2\), there exist open convex sets \(U_\lambda,V_\lambda\ni\bar x+\lambda\bar v\) on which \(\env f\) and \(\env g\) are convex, respectively. Therefore, for every \(0<\lambda<\bar\lambda\defeq\min(\bar\lambda_1,\bar\lambda_2)\),
    \(
    \env\varphi^\alpha_\lambda
    =
    \alpha\env f+(1-\alpha)\env g
    \)
    is convex on the open convex set \(W_\lambda\defeq U_\lambda\cap V_\lambda\ni\bar x+\lambda\bar v\), and \cite[Proposition 5.2(a)]{chen2020proximal} yields that \(\varphi_\lambda^\alpha\) is prox-regular at \(\bar x\) for \(\bar v\in\lsubdiff\varphi_\lambda^\alpha(\bar x)\).
   Appealing to \cref{thm:variational cvx:3} once again completes the proof.
\end{proof}

Using \cref{thm: pa vc}, we can construct parameterized families of variationally convex functions.
	\begin{example}
		Let \(f(x)=1\), for \(x\neq0\), and \(f(0)=0\), and let \(g(x)=\tfrac{x^2}{2}\). It holds
			\begin{align*}
				(\forall \lambda>0) (\forall 0\leq \alpha\leq1) \quad \varphi_\lambda^\alpha(x)
			=
				\begin{cases}
					\left(\tfrac{1+\lambda}{\alpha}-1\right)\tfrac{x^2}{2\lambda},
					&\text{ if }|x|\leq\tfrac{\alpha}{1+\lambda}\sqrt{2\lambda};\\
					-\tfrac{x^2}{2\lambda}+\sqrt{\tfrac{2}{\lambda}}|x|-\tfrac{\alpha}{1+\lambda},
					&\text{ if }\tfrac{\alpha}{1+\lambda}\sqrt{2\lambda}<|x|<\tfrac{1+(1-\alpha)\lambda}{1+\lambda}\sqrt{2\lambda};\\
					\left(\tfrac{1+\lambda}{1+(1-\alpha)\lambda}-1\right)\tfrac{x^2}{2\lambda}+1-\alpha,
					&\text{ otherwise,}
				\end{cases}
			\end{align*}
			where we adopt the convention that \(1/0=+\infty\) and \(0\cdot(+\infty)=0\) when \(\alpha=0\).
			Then,  for every sufficiently small \(\lambda\) and every \(0\leq\alpha\leq 1\), the proximal average \(\varphi_\lambda^\alpha\) is variationally convex at \(\bar x=0\) for \(\bar v=0\).
	\end{example}
	\begin{proof}
		It is easy to see that \(\env f(w)=\tfrac{w^2}{2\lambda}\), for \(|w|\leq\sqrt{2\lambda}\), and \(\env f(w)=1\), otherwise, whereas \(\env g(w)=\tfrac{w^2}{2(1+\lambda)}\).
		Hence \(h_x(w)\defeq-(1-\alpha)\env f(w)-\alpha\env g(w)+\tfrac{1}{2\lambda}(w-x)^2\) satisfies
		\begin{align*}
			h_x(w)
		=
			\begin{cases}
				p_x(w)\defeq\tfrac{\alpha}{2\lambda(1+\lambda)}w^2-\tfrac{x}{\lambda}w+\tfrac{x^2}{2\lambda},
				&\text{ if }|w|\leq\sqrt{2\lambda};\\
				q_x(w)\defeq\tfrac{1+(1-\alpha)\lambda}{2\lambda(1+\lambda)}w^2-\tfrac{x}{\lambda}w+\tfrac{x^2}{2\lambda}-(1-\alpha),
				&\text{ otherwise.}
			\end{cases}
		\end{align*}
		Define \(w_1^*\defeq\argmin_w p_x(w)=\tfrac{1+\lambda}{\alpha}x\), when \(0<\alpha\leq1\), and \(w_1^*:=\sgn(x)\sqrt{2\lambda}\), when \(\alpha=0\), and \(w_2^*\defeq\argmin_w q_x(w)=\tfrac{1+\lambda}{1+(1-\alpha)\lambda}x\), for \(0\leq\alpha\leq1\).
		Simple calculation shows that
		\[
		|w_1^*|\leq\sqrt{2\lambda}
		\Leftrightarrow
		|x|\leq\tfrac{\alpha\sqrt{2\lambda}}{1+\lambda}
		\text{ and }
		|w_2^*|\leq\sqrt{2\lambda}
		\Leftrightarrow
		|x|\leq\tfrac{1+(1-\alpha)\lambda}{1+\lambda}\sqrt{2\lambda},
		\]
		from which we derive easily that
		\begin{align*}
				\argmin_{w\in\R} h_x(w)
			=
				\begin{cases}
						w_1^*,
					&\text{ if } 
						|x|\leq\tfrac{\alpha}{1+\lambda}\sqrt{2\lambda};\\
						\sgn(x)\sqrt{2\lambda},
					&\text{ if }
						\tfrac{\alpha}{1+\lambda}\sqrt{2\lambda}<|x|<\tfrac{1+(1-\alpha)\lambda}{1+\lambda}\sqrt{2\lambda};\\
					w_2^*,
					&\text{ otherwise.}
				\end{cases}			
		\end{align*}
		The proximal average formula then follows from the fact that \(\varphi_\lambda^\alpha=-\inf_w h_x(w)\).
		Since both \(f\) and \(g\) are variationally convex at \(\bar x\) for \(\bar v\), \cref{thm: pa vc} implies that so is \(\varphi_\lambda^\alpha\) for \(\lambda>0\) sufficiently small.
	\end{proof}
% \end{comment}

\begin{remark}[calculus rules of variational strong convexity]
	While the calculus rules developed in this section are formulated in terms of variational convexity, they can be readily extended to the setting of variational strong convexity.
\end{remark}

\section{Applications to nonlinear programming}\label{sec:nonlinear}

In this section, we demonstrate how the calculus rules developed in \cref{sec:calculus} enable the identification of variational convexity in constrained nonlinear programming problems 
\begin{equation}\label{eq:nlp}
	\tag{NLP}
	\begin{aligned}
		\min_{x\in\Rn}& \ f(x)\\
		\text{s.t.~}&g_i(x)\leq0,i=1,\ldots,m,\\
		&g_i(x)=0,i=m+1,\ldots,p,
	\end{aligned}
\end{equation}
where \(\func{f}{\Rn}{\RE}\) is proper and lower semicontinuous, and \(\func{g_i}{\Rn}{\R}\) is \(C^1\) for \(i=1,\ldots,p\).
The problem \eqref{eq:nlp} can be equivalently written as
\begin{equation}\label{eq:nlp'}
	\min_{x\in\Rn}
	\varphi(x)=f(x)+\delta_{C}(G(x)),
\end{equation}
where \(G(x)=(g_1(x),\ldots,g_p(x))\) and \(C\subseteq\R^{p}\) is defined by
\[
	C=\set{(u_1,\ldots,u_p)}[u_i\leq0~\forall i=1,\ldots,m,u_i=0~\forall i=m+1,\ldots,p].
\]
Recent interest in variational convexity has largely been driven by its implications for nonlinear programming, starting with Rockafellar's seminal papers on local optimality in generalized nonlinear programming problems \cite{rockafellar2023augmented} and on the local linear convergence of augmented Lagrangian methods \cite{rockafellar2023convergence}.
Considering the general nonlinear programming problem
\begin{equation}\label{eq:gnlp}
    \min_{x\in\Rn}f(x)+h(G(x)),
\end{equation}
where \(f,G\) are twice continuously differentiable and \(\func{h}{\R^m}{\RE}\) is polyhedral convex, Rockafellar \cite[Theorem 4]{rockafellar2023augmented} showed that the variational strong convexity of the augmented objective \((x,u)\mapsto f(x)+h(G(x)+u)+\tfrac{r}{2}\norm{u}^2\), for some \(r>0\), amounts to the Hessian of the associated Lagrangian being positive definite relative to a subspace, a condition that can be viewed as the classical strong second-order sufficient condition in nonlinear programming. Subsequently, variational convexity and variational strong convexity of \(\varphi\) in the \(C^2\) setting of \eqref{eq:nlp'} have been studied by Khanh, Mordukhovich, and Phat in \cite[Theorem 7.3]{Khanh2023variational}, who derived analogous characterizations in terms of the Hessian of the associated Lagrangian; see also \cite[Theorem 9.3]{khanh2025second}. Moreover, \cite[Theorem 3.3]{wang2023strong} considered the optimization problem \eqref{eq:gnlp} with \(h=\delta_{\mathbb{S}^m_+}\), where \(\mathbb{S}^m_+\) denotes the cone of positive semidefinite \(m\times m\) matrices, and showed that variational strong convexity holds if and only if a similar Hessian condition on the Lagrangian is satisfied.

Although these results investigate variational (strong) convexity in different settings, they all rely on conditions involving the Hessian of the associated Lagrangian, thereby requiring both the objective function \(f\) and constraint functions \(g_i\) to be twice continuously differentiable. Consequently, these approaches are not applicable to problems involving nonsmooth penalties, where the objective function  \(f\) in \eqref{eq:nlp} takes the form
\[
f(x)=l(x)+\text{nonsmooth penalty},
\]
for some function \(l\). Optimization problems of this type arise frequently in practice. A prototypical example is
\(
	f(x)=\tfrac{1}{2}\norm{Ax-b}^2+r\norm{x}_0,
\)
where \(A\in\R^{m\times n}\), \(b\in\R^m\), and \(r>0\). This objective appears in \(\ell_0\)-penalized least-squares problems. From a broader perspective, classical second-order sufficient conditions in nonlinear programming guarantee local optimality, whereas variational convexity yields local optimality using only first-order information.
This naturally leads to the following question: can variational convexity be employed to derive first-order sufficient conditions for local optimality?

We address this question by investigating the variational convexity of \(\varphi\) in~\eqref{eq:nlp'} without relying on second-order properties. Our analysis is based on the calculus rules for variational convexity developed in \cref{sec:calculus}. As a consequence, the objective function \(f\) may be nonsmooth, while only continuous differentiability of the constraint functions \(g_i\), is required, thereby encompassing optimization problems with nonsmooth penalties. The trade-off is that the equality constraints in \eqref{eq:nlp} are required to be affine.

Throughout the remainder of this section, let
\[
    \mathcal{A}(\bar x):=\set{1\leq i\leq m}[g_i(\bar x)=0]
    \text{ and }
    \mathcal J(\bar x):=\mathcal{A}(\bar x)\cup\{m+1,\ldots,p\},
\]
be the set of active inequality constraints and the set of active constraints at \(\bar x\), respectively.
% Recall that
% \(
% 	\mathcal{I}(\bar x)=\set{i\in\{1,\ldots,n\}}[\bar x_i\neq0]
% \).
We say that the\emph{ Positive Linear Independence Constraint Qualification} \eqref{eq:PLICQ} for \eqref{eq:nlp} holds at \(\bar x\) if
\begin{equation}\label{eq:PLICQ}
	\tag{PLICQ}
	% \begin{cases}
	% 	\{\nabla g_i(\bar x)\}_{i=m+1}^p\text{ are linearly independent},\\
	% 	(\exists d\in\Rn)~
	% 	\nabla g_i(\bar x)^Td<0 \text{ for }i\in\mathcal{A}(\bar x)\text{ and }\nabla g_i(\bar x)^Td=0\text{ for }i=m+1,\ldots,p.
	% \end{cases}
		\left[
			\displaystyle\sum_{i\in \mathcal J(\bar x)}u_i\nabla g_i(\bar x)=0
		\text{ and }
			u_i\geq0~\forall i\in\mathcal{A}(\bar x)
		\right]
	\Rightarrow
		[u_i=0~\forall i\in \mathcal J(\bar x)].
\end{equation}
We also need the following separable sum rule for variational convexity.
\begin{lemma}\emph{\cite[Proposition 5.14]{luo2024level}}\label{lem:sep}
	Let \(\func{f_i}{\R^{n_i}}{\RE}\) be proper and lower semicontinuous functions for \(i=1,\ldots,m\).
	If \(f_i\) is variationally convex at \(\bar x_i \in \dom f_i\) for \(\bar v_i\in\lsubdiff f_i(\bar x_i)\) for \(i=1,\ldots,m\), then \(f(x)=\sum_{i=1}^mf_i(x_i)\) is variationally convex \(\bar x=(\bar x_1,\ldots,\bar x_m)\) for \(\bar v=(\bar v_1,\ldots,\bar v_m)\in\lsubdiff f(\bar x)\).
\end{lemma}

The nonlinear composition rule \cref{thm:nonlinear comp} allows us to derive the following result concerning the variational convexity of \(\varphi\) in \eqref{eq:nlp'} in a general setting.

\begin{theorem}\label{thm:nonlinear vc}
	Let \(\bar x \in \R^n\) be a critical point for the optimization problem \eqref{eq:nlp}, that is, let $\bar x$ satisfy \(0\in\lsubdiff\varphi(\bar x)\), where \(\varphi\) is defined as in \eqref{eq:nlp'}. If
	\begin{enumerateq}
		\item\label{thm:nonlinear vc:1}
		% \(f\) is variationally convex at \(\bar x\) for every \(u\in\lsubdiff f(\bar x)\);
		\(f\) is variationally convex at \(\bar x\) for every \(\bar v\in\lsubdiff f(\bar x)\) with the property that
		\(
            -\bar v\in\nabla G(\bar x)^*N_C(G(\bar x));
        \)
		\item\label{thm:nonlinear vc:2}
			\eqref{eq:PLICQ} holds at \(\bar x\) and
			\begin{equation}\label{eq:nonlinear vc cq}
				\partial^\infty f(\bar x)\cap\left(-\nabla G(\bar x)^*N_C(G(\bar x))\right)=\{0\};
			\end{equation}
		\item\label{thm:nonlinear vc:3}
			\((\forall i\in \mathcal{A}(\bar x))\)~\(g_i\) is locally convex around \(\bar x\) and \((\forall m+1\leq i\leq p)\) \(g_i\) is affine,
	\end{enumerateq}
then \(\varphi\) is variationally convex at \(\bar x\) for \(0\in\lsubdiff \varphi(\bar x)\).
	Consequently, \(\bar x\) is a local minimizer of \eqref{eq:nlp}.
\end{theorem}
\begin{proof}
Define $g(x,y):=f(x)+\delta_C(y)$, where the set \(C\) is given in \eqref{eq:nlp'}, and let $F(x):=(x,G(x))$. Then \(\varphi=g\circ F\). We will establish the desired variational convexity by verifying all the assumptions of \cref{thm:nonlinear comp}. Since \(\delta_C\) is convex, it is variationally convex at \(G(\bar x)\) for every subgradient in \(\partial \delta_C(G(\bar x))=N_C(G(\bar x))\). By \cref{lem:sep} and (i), it follows that \(g\) is variationally convex at \(F(\bar x)\) for every
$(\bar v,\bar u)\in\lsubdiff g(F(\bar x))
=\lsubdiff f(\bar x)\times N_C(G(\bar x))$
satisfying $0=\nabla F(\bar x)^*(\bar v,\bar u)$.

We next verify condition \eqref{eq:chain cq}. To this end, let
$(v,u)\in\partial^\infty g(F(\bar x))
\subseteq
\partial^\infty f(\bar x)\times N_C(G(\bar x))$ satisfy 
$\nabla F(\bar x)^*(v,u)=0$.
The inclusion above follows from \cite[Proposition 10.5]{rockafellar_variational_1998}. Then $v+\nabla G(\bar x)^*u=0$,
and hence condition \eqref{eq:nonlinear vc cq} implies that
$v=0$ and $\nabla G(\bar x)^*u=0$. To show that \(u=0\) as well, recall from \cite[Lemma 7.2]{Khanh2023variational} that
\begin{equation*} \ker\nabla G(\bar x)^*\cap N_C(G(\bar x))=\set{u\in\R^{p}} [\sum_{i\in \mathcal J(\bar x)}u_i\nabla g_i(\bar x)=0,u_i\geq0~\forall i\in\mathcal{A}(\bar x) \text{ and }u_i=0~\forall i\notin \mathcal J(\bar x) ]. \end{equation*}
Therefore, condition \eqref{eq:PLICQ} yields $\ker \nabla G(\bar x)^* \cap N_C(G(\bar x))
=
\{0\}$,
which implies that \(u=0\). This verifies condition \eqref{eq:chain cq}.
	
	By continuity of the functions \(g_i\) at \(\bar x\), there exists \(\varepsilon>0\) such that, for every \(x\) satisfying \(\|x-\bar x\|<\varepsilon\),
\begin{equation}\label{eq:active indice}
(\forall i\in\{1,\ldots,m\}\setminus\mathcal{A}(\bar x))
\qquad
g_i(x)
<
-\max_{j\in\{1,\ldots,m\}\setminus\mathcal{A}(\bar x)} g_j(\bar x)
\le 0,
\end{equation}
and consequently $\mathcal{A}(x)\subseteq\mathcal{A}(\bar x)$. Define
\[
T:=\left\{
(x,v)\in\operatorname{gph}\lsubdiff\varphi
\;\middle|\;
\|x-\bar x\|<\varepsilon
\right\},
\]
and suppose that \(g_i\) is convex on an open convex neighborhood \(C_i\ni\bar x\) for every \(i\in\mathcal{A}(\bar x)\). To verify the remaining condition \eqref{eq:Gu convex}, we shall establish the following stronger property:
    \begin{equation}
    (\forall u\in \cup_{(x,v)\in T}S(x,v))\quad
        \ip{u}{F(\cdot)}
    \text{ is convex on }C\defeq\cap_{i\in \mathcal{A}(\bar x)}C_i,
    \end{equation}
    where
    \(
		\ffunc{S}{\Rn\times\Rn}{\R^m}:(x,v)\mapsto\set{u\in\lsubdiff g(F(x))}[v=\nabla F(x)^*u].
	\)
Pick $(x,v)\in T$ and $u = (u_1, u_2) \in S(x,v)$, that is, \(u=(u_1,u_2)\in\lsubdiff g(F(x))
=\lsubdiff f(x)\times N_C(G(x))\) and $v=\nabla F(x)^*u$. Consequently, since \(u_2\in N_C(G(x))\), we have $
(u_2)_i\ge 0 \text{ for every } i\in\mathcal A(x)$, and $(u_2)_i=0 \text{ for every } i\notin\mathcal A(x)$.
Therefore, the function
\[
\langle u,F(\cdot)\rangle
=
\langle u_1,\cdot\rangle
+
\langle u_2,G(\cdot)\rangle
=
\underbrace{
\langle u_1,\cdot\rangle
+
\sum_{i=m+1}^{p}(u_2)_i g_i(\cdot)
}_{\text{affine}}
+
\underbrace{
\sum_{i\in\mathcal A(x)}(u_2)_i g_i(\cdot)
}_{\text{convex on } C}
\]
is convex on \(C\). Indeed, by \eqref{eq:active indice} we have
\(\mathcal A(x)\subseteq\mathcal A(\bar x)\), and \(g_i\) is convex on \(C\) for every
\(i\in\mathcal A(\bar x)\). This establishes condition \eqref{eq:Gu convex}. Hence, all the assumptions of \cref{thm:nonlinear comp} are satisfied, and therefore $\varphi$ is variationally convex at \(\bar x\) for \(0\in\lsubdiff \varphi(\bar x)\).
\end{proof}
We now turn to concrete applications.
The next result concerns a jump-penalized optimization problem with constraints.
A prototypical example is the unconstrained jump-penalized least-squares problem
\begin{equation}\label{potts least square}
\min_{x\in\mathbb{R}^n}
\frac{1}{n}\sum_{i=1}^n (x_i-b_i)^2
+
\card\bigl\{\,1\le i\le n-1 \mid x_i\neq x_{i+1}\,\bigr\},
\end{equation}
where \(b\in\mathbb{R}^n\) is a given signal. Problem \eqref{potts least square} arises in a variety of applications, including statistical mechanics \cite{potts1952some}, image reconstruction and segmentation \cite{storath2015joint}, and the approximation of regression functions \cite{boysen2009consistencies}. Building on the variational convexity of the jump function established in \cref{eg:jump vc}, we now analyze a constrained counterpart of \eqref{potts least square}.

\begin{corollary}\label{thm:jump constrained vc}
	Consider the constrained optimization problem penalized with the jump function
	\begin{equation}\label{constrained potts}
		\begin{aligned}
			\min_{x\in\Rn}&~f(x)=l(x)+\card\set{1\leq i\leq n-1}[x_i\neq x_{i+1}]\\
			\text{\rm s.t.~}&g_i(x)\leq0,i=1,\ldots,m,\\
			&g_i(x)=0,i=m+1,\ldots,p,
		\end{aligned}
	\end{equation}
	where \(\func{l}{\Rn}{\R}\) is convex, and \(\func{g_i}{\Rn}{\R}\) is \(C^1\) for every \(i=1,\ldots,p\). Let \(\bar x \in \R^n\) be a critical point  for the optimization problem \eqref{constrained potts}, that is, let $\bar x$ satisfy \(0\in\lsubdiff\varphi(\bar x)\), where \(\varphi\) is defined as in \eqref{eq:nlp'}. If
	\begin{enumerateq}
		\item 
		\eqref{eq:PLICQ} holds at \(\bar x\) and 
			\begin{equation}\label{eq: potts cq}
			\set{D^*v}
			[
					v_i=0\text{ if }\bar x_i\neq \bar x_{i+1}
			]
			\cap
			\left(-\nabla G(\bar x)^*N_C(G(\bar x))\right)
		=
			\set{0},
		\end{equation}
		where
    \[
	D=
	\begin{bmatrix}
		1 & -1 & 0 & \cdots & 0\\
		0 & 1 & -1 & \cdots & 0\\
		\vdots & \vdots & \vdots & \ddots & \vdots\\
		0 & 0 & 0 & \cdots & -1
	\end{bmatrix} \in\R^{(n-1)\times n}
	\]
is the matrix representation of the difference operator;
		\item
		\((\forall i\in\mathcal{A}(\bar x))\) \(g_i\) is  locally convex around \(\bar x\)  and \((\forall m+1\leq i\leq p)\) \(g_i\) is affine,
	\end{enumerateq}
then \(\varphi\) is variationally convex at \(\bar x\) for \(0\in\lsubdiff\varphi(\bar x)\). Consequently. \(\bar x\) is a local minimizer of the optimization problem~\eqref{constrained potts}.
\end{corollary}
\begin{proof}
	Denote by \(J(x)=\card\set{1\leq i\leq n-1}[x_i\neq x_{i+1}]\).
	Recall from \cref{eg:jump vc} that
	\begin{align*}
		\lsubdiff J(\bar x)
	= \partial^\infty J(\bar x) = 
		\set{D^*v}[v_i=0\text{ if }\bar x_i\neq \bar x_{i+1}].
	\end{align*}
We verify the assumptions of \cref{thm:nonlinear vc}. By \cref{eg:jump vc},  \(J\) is variationally convex at every \(\bar x\in \R^n\) for every subgradient at \(\bar x\). Since \(l:\R^n\to\R\) is convex and finite-valued, we have $\partial^\infty l(\bar x) = \{0\}$. Therefore, the variational convexity sum rule in \cref{thm:sum rule} implies that \(f\) is variationally convex at \(\bar x\) for every subgradient of \(f\) at \(\bar x\).
    
Moreover, the singular subdifferential sum rule yields	\[
	\partial^\infty f(\bar x)
	\subseteq
		\partial^\infty l(\bar x)
		+
		\partial^\infty J(\bar x) = \set{D^*v}[v_i=0\text{ if }\bar x_i\neq \bar x_{i+1}].
	\]
Consequently, condition \eqref{eq:nonlinear vc cq} follows from \eqref{eq: potts cq}. Hence all assumptions of \cref{thm:nonlinear vc} are satisfied, and the asserted conclusion follows.
\end{proof}

% \begin{example}
% 	Let \(l(x_1,x_2)=\tfrac{1}{2}(x_1^2+x_2^2)\) and let \(h(x)=0\) for \(x\leq0\) and \(h(x)=1\) otherwise.
% 	Consider
% 	\[
% 	\begin{aligned}
% 		\min_{(x_1,x_2)\in\R^2}~
% 	&f(x_1,x_2)=l(x_1,x_2)+h(x_1)+h(x_2)\\
% 	\text{s.t. }
% 	&g_1(x_1,x_2)=x_1+x_2-1\leq0\\
% 	&g_2(x_1,x_2)=x_1-1=0.
% 	\end{aligned}
% 	\]
% 	Let \(\bar x=(1,0)\).
% 	Then
% 	\(
% 	\mathcal{A}(\bar x)=\set{1}
% 	\),
% 	\(\mathcal{J}(\bar x)=\set{1,2}\),
% 	\(
% 	\mathcal{I}(\bar x)=\set{1}
% 	\),
% 	and
% 	\(
% 	\nabla l(\bar x)+u_1\nabla g_1(\bar x)+u_2\nabla g_2(\bar x)
% 	=
% 	(1+u_1+u_2,u_1)
% 	\) for every \(u_1,u_2\in\R\).
% 	Obviously \(\bar x\) is a stationary point.
% 	We next show it's a local minimizer by verifying \eqref{eq:easy jump cq}.\
% 	Pick \(u_1,u_2\in\R\) such that
% 	\begin{equation*}
% 		v=(u_1+u_2,u_1)\in\R^2_-,
% 		v_1=0,
% 		u_1\geq0.
% 	\end{equation*}
% 	Then \(u_1=0\) and consequently it must hold that \(u_2=0\), justifying \eqref{eq:easy jump cq}.\
% 	Hence all assumptions of \cref{thm:jump constrained problem vc} hold, and therefore \(\varphi=f+\delta_\Omega\) variationally convex at \(\bar x\) for \(0\in\lsubdiff\varphi(\bar x)\), where \(\Omega=\set{x\in\R^2}[x_1+x_2-1\leq0,x_1=1]\), and \(\bar x\) is a local minimizer of \(\varphi\).
% \end{example}

We now turn to a class of constrained optimization problems involving penalties based on the Heaviside function, defined by
\[
\func{h}{\R}{\R}:
x\mapsto
\begin{cases}
0,& x\le 0,\\
1,& x>0.
\end{cases}
\]
The Heaviside function arises in numerous applications, most notably in image segmentation; see, for example, \cite[Section~2]{chan2001active}. We begin by showing that it is variationally convex.

\begin{example}\label{eg:heavy vc}
Let
\[
h(x)=
\begin{cases}
0,& x\le 0,\\
1,& x>0.
\end{cases}
\]
Then \(h\) is variationally convex at every point and with respect to every subgradient.
\end{example}

\begin{proof}
Since \(h\) is locally constant away from the origin, it suffices to verify variational convexity at \(\bar x=0\) for arbitrary \(\bar v\in\lsubdiff h(0)\). Observe that
\[
\rsubdiff h(x)=\lsubdiff h(x)=
\begin{cases}
\{0\},& x\neq 0,\\
\R_+,& x=0.
\end{cases}
\]
We first consider the case \(\bar v=0\). Let \(0<\varepsilon<1\), and define
$U=V=(-\varepsilon,\varepsilon)$. Since
$U_\varepsilon
=
\{x\in U | h(x)<h(\bar x)+\varepsilon\}
=
(-\varepsilon,0]$,
we obtain $(U_\varepsilon\times V)\cap \gra \lsubdiff h
=
(-\varepsilon,0]\times\{0\}$,
which is clearly a monotone set. Hence \(h\) is variationally convex at \(\bar x\) with respect to \(\bar v=0\) by \cref{thm:variational cvx:2}.

Next, let \(\bar v>0\). Choose \(\varepsilon>0\) sufficiently small so that
$ 0\notin (\bar v-\varepsilon,\bar v+\varepsilon)$,
and set $V=(\bar v-\varepsilon,\bar v+\varepsilon)$, and
$U=(-\varepsilon,\varepsilon)$.
Then $(U\times V)\cap \gra \lsubdiff h
=
\{0\}\times (\bar v-\varepsilon,\bar v+\varepsilon)$,
which is also a monotone set. Another application of \cref{thm:variational cvx:2} therefore shows that \(h\) is variationally convex at \(\bar x\) with respect to \(\bar v\).

Since \(\bar v\in\lsubdiff h(0)\) was arbitrary, the claim follows.
\end{proof}

\begin{corollary}\label{thm:jump constrained problem vc}
Consider the constrained optimization problem penalized with the Heaviside function
\begin{equation}\label{jump regularized problem}
		\begin{aligned}
			\min_{x\in\Rn}&f(x)=l(x)+r\sum_{i=1}^nh(x_i)\\
			\text{\rm s.t.~}&g_i(x)\leq0,i=1,\ldots,m,\\
			&g_i(x)=0,i=m+1,\ldots,p,
		\end{aligned}
	\end{equation}
	where \(\func{l}{\Rn}{\R}\) is convex, \(\func{g_i}{\Rn}{\R}\) is \(C^1\) for every \(i=1,\ldots,p\), 
\[
\func{h}{\R}{\R}:
x\mapsto
\begin{cases}
0,& x\le 0,\\
1,& x>0.
\end{cases}
\]
is the Heaviside function, and $r >0$. Let \(\bar x \in \R^n\) be a critical point  for the optimization problem \eqref{jump regularized problem}, that is, let $\bar x$ satisfy \(0\in\lsubdiff\varphi(\bar x)\), where \(\varphi\) is defined as in \eqref{eq:nlp'}. If
   \begin{enumerateq}
		\item\label{thm:jump constrained problem vc:cq}
		\eqref{eq:PLICQ} holds at \(\bar x\) and
			\begin{equation}\label{eq: jump cq}
			\set{v\in\R^n_+}[v_j=0~\forall  j\in\mathcal I(\bar x)]\cap\left(-\nabla G(\bar x)^*N_C(G(\bar x))\right)
		=
			\set{0};
		\end{equation}
		\item
		\((\forall i\in\mathcal{A}(\bar x))\) \(g_i\) is locally convex around \(\bar x\)  and \((\forall m+1\leq i\leq p)\) \(g_i\) is affine,
	\end{enumerateq}
then \(\varphi\) is variationally convex at \(\bar x\) for \(0\in\lsubdiff\varphi(\bar x)\). Consequently. \(\bar x\) is a local minimizer of the optimization problem \eqref{jump regularized problem}.
\end{corollary}
\begin{proof}
Write $H(x):=r\sum_{i=1}^n h(x_i)$
for brevity. It is straightforward to verify that
\[
\partial^\infty H(\bar x)
=
\rsubdiff H(\bar x)
=
\lsubdiff H(\bar x)
=
\{v\in\R_+^n | v_i=0 \ \forall i\in\mathcal I(\bar x)\}.
\]
We next show that \(f\) is variationally convex at \(\bar x\) with respect to every subgradient of \(f\) at \(\bar x\).  Since \(l:\R^n\to\R\) is convex and finite-valued, we have $\partial^\infty l(\bar x) = \{0\}$.  Hence, by the singular subdifferential sum rule,
\[
\partial^\infty f(\bar x)
\subseteq
\partial^\infty l(\bar x)+\partial^\infty H(\bar x)
=
\{v\in\R_+^n | v_i=0 \ \forall i\in\mathcal I(\bar x)\}.
\]
Consequently, condition \eqref{eq: jump cq} implies that $\partial^\infty f(\bar x)
\cap
\bigl(-\nabla G(\bar x)^*N_C(G(\bar x))\bigr)
=
\{0\}$.

Furthermore, \cref{eg:heavy vc} and the separable sum rule for variational convexity (see \cref{lem:sep}) imply that \(H\) is variationally convex at every point with respect to every subgradient. Applying the variational convexity sum rule \cref{thm:sum rule}, we conclude that \(f=l+H\) is variationally convex at \(\bar x\) with respect to every subgradient of \(f\) at \(\bar x\).

Therefore, all assumptions of \cref{thm:nonlinear vc} are satisfied, and the desired conclusion follows.
\end{proof}
\begin{remark}
Assume that the following constraint qualification holds at \(\bar x\):
\begin{equation}\label{eq:easy jump cq}
			\left[
				v\defeq\sum_{i\in \mathcal J(\bar x)}u_i\nabla g_i(\bar x)\in\R^n_-, \
				v_j=0~\forall j\in \mathcal I(\bar x)
				\text{ and }
				u_i\geq0~\forall i\in\mathcal{A}(\bar x)
			\right]
			\Rightarrow
				\left[u_i=0~\forall i\in \mathcal J(\bar x)\right]
	\end{equation}
Then both assumptions of \cref{thm:jump constrained problem vc:cq} are satisfied.

Clearly, \eqref{eq:easy jump cq} implies \eqref{eq:PLICQ}. Now let \(\tilde v\) belong to the intersection appearing in \eqref{eq: jump cq}. To verify \eqref{eq: jump cq}, it suffices to show that \(\tilde v = 0\). Define $v \defeq -\tilde v
    = \sum_{i\in\mathcal J(\bar x)} u_i \nabla g_i(\bar x)$
for some \(u \in N_C(G(\bar x))\). Since \(\tilde v \in \mathbb{R}^n_+\) and
\(\tilde v_j = 0\) for every \(j \in \mathcal I(\bar x)\), it follows that
\(v \in \mathbb{R}^n_-\) and \(v_j = 0\) for every \(j \in \mathcal I(\bar x)\). Moreover, \(u \in N_C(G(\bar x))\) implies that
\(u_i \ge 0\) for every \(i \in \mathcal A(\bar x)\).
Therefore, all assumptions of \eqref{eq:easy jump cq} are satisfied, and hence
\(u_i = 0\) for all \(i \in \mathcal J(\bar x)\).
Consequently, \(v = 0\), and thus \(\tilde v = 0\).
\end{remark}

\section{Conclusions}
This paper investigates characterizations and calculus rules for variational convexity, together with their applications to nonlinear programming. We characterize the variational convexity of a function \(\func{f}{\mathbb{R}^n}{\overline{\mathbb{R}}}\) in terms of its proximal hulls and epigraph, thereby providing deeper insight into this notion. We further establish calculus rules showing that variational convexity is preserved under fundamental operations in optimization, including nonlinear and linear compositions, summation, and proximal averaging. These rules facilitate the identification of variationally convex functions without the need to verify their characterizations directly.

In addition to yielding new examples of variationally convex functions, our results lead to new local optimality conditions for nonlinear programming problems. Compared with existing results in the literature, our approach based on calculus rules allows for nonsmooth objective functions while requiring only \(C^1\)-smoothness of the constraints.

\section*{Acknowledgment}
The research of ZW was supported by a Postdoctoral Fellowship from the Natural Sciences and Engineering Research Council of Canada.
The authors thank Xianfu Wang for many insightful discussions.

\bibliographystyle{siamplain}
\bibliography{biblio}

\appendix
\section{Proof of \cref{lem:chain rule cq}}
\begin{newitemize}
	\item``\ref{lem:chain rule cq::1}''
	In view of the subdifferential chain rule \cite[Theorem 10.6]{rockafellar_variational_1998}, it suffices to show that there exists \(\varepsilon>0\) such that, for every \(x \in \Rn\) satisfying \(\norm{x-\bar x}<\varepsilon\) and \(f(x)<f(\bar x)+\varepsilon\), it holds
	\begin{equation*}
		\ker\nabla F(x)^*\cap\partial^\infty g(F(x))=\set{0}.
	\end{equation*}
	Suppose, to the contrary, that there exist sequences \(\seq{x_k}\) and \(\seq{u_k}\) such that \(x_k\to \bar x\), \(f(x_k)\leq f(\bar x)+1/k\), and \(0\neq u_k\in\partial^\infty g(F(x_k))\) with \(\nabla F(x_k)^*u_k=0\) for every \(k\in\N\). Since \(f\) is lower semicontinuous at \(\bar x\), it follows that \(g(F(x_k))=f(x_k)\to f(\bar x)=g(F(\bar x))\) as $k \to +\infty$.
	Then
	\[
	(\forall k\in\N)\quad
	\tfrac{u_k}{\norm{u_k}}\in\partial^\infty g(F(x_k))\text{ and }\nabla F(x_k)^*\tfrac{u_k}{\norm{u_k}}=0.
	\]
	Taking subsequences if necessary, we may assume that \(\tfrac{u_k}{\norm{u_k}}\to u\neq0\) as  $k \to +\infty$. Then, the outer semicontinuity of \(\partial^\infty g\) with respect to attentive convergence yields \(u\in\partial^\infty g(F(\bar x))\) and \(\nabla F(\bar x)^*u=0\). This contradicts \eqref{eq:chain rule cq}.
	\item``\ref{lem:chain rule cq::2}''
	The outer semicontinuity of \(S\) follows from that of \(\partial g\).
	Suppose that \(S\) is not locally bounded at \((\bar x,\bar v)\) with respect to \(f\)-attentive convergence, say there exist sequences \((x_k,v_k)\to (\bar x,\bar v)\) with \(f(x_k)\to f(\bar x)\) and \(u_k\in S(x_k,v_k)\) such that \(\norm{u_k}\to+\infty\) as $k \to +\infty$.
	Then, by taking subsequences if necessary, we may assume that 
	\(
	\tfrac{u_k}{\norm{u_k}}\to u
	\) as $k \to +\infty$, implying
	\(0\neq u\in\partial^\infty g(F(\bar x))\).
	But
	\[
	\nabla F(x_k)^*\frac{u_k}{\norm{u_k}}=\tfrac{v_k}{\norm{u_k}}\to 0
	\text{ as }k\to+\infty, \ \text{hence,} \
	\nabla F(\bar x)^*u=0.
	\]
This violates \eqref{eq:chain rule cq}. It follows from \cite[Theorem 5.19]{rockafellar_variational_1998} that \(S\) is usc at \((\bar x,\bar v)\) with respect to \(f\)-attentive convergence.
	\qed
\end{newitemize}
\end{document}